\documentclass{article}

\usepackage[titletoc,toc,title]{appendix}

\usepackage{amssymb,amsfonts}
\usepackage{amsthm}
\usepackage{amsmath}
\usepackage{graphicx} 
\usepackage{amsfonts}
\usepackage{amssymb} 
\usepackage{latexsym}
\usepackage{pifont}
\usepackage{graphics,graphicx,psfrag}
\usepackage{amscd}
\usepackage{epsfig}
\usepackage[all]{xy}
\usepackage{float}
\usepackage{enumitem}
\usepackage{cancel}

\newtheorem{theorem}{Theorem}[section]
\newtheorem{proposition}{Proposition}[section]
\newtheorem{definition}{Definition}[section]
\newtheorem{remark}{Remark}[section]
\newtheorem{lemma}{Lemma}[section]
\newtheorem{example}{Example}[section]
\newtheorem{corollary}{Corollary}[section]

\usepackage[hang,flushmargin]{footmisc} 
\usepackage[usenames,dvipsnames]{color}
\usepackage{tikz}
\usepackage{xcolor}
\usepackage{parskip}
\usetikzlibrary{matrix,arrows}
\usepackage{enumitem}
\usepackage{fullpage}
\usepackage{setspace}
\setlist{nolistsep}

\setlist[itemize]{parsep=\baselineskip}

\newcommand{\Modl}{ \mbox{}_R{\rm{\bf Mod}} }

\newcommand{\Cadl}{ {\rm {\bf Ch}}(\mbox{}_R{\rm {\bf Mod}}) }
\newcommand{\Complexes}{ {\rm {\bf Ch}}(\mathcal{C}) }   % The category of chain complexes over an Abelian category.

\begin{document}

\thispagestyle{empty}

\title{Adjointness properties for relative extensions \\ of disk and sphere chain complexes}

\author{Marco A. P\'erez B. \\ Department of Mathematics \\ Massachusetts Institute of Technology }

\maketitle

%%%%%%%%%%%%%%%%%%%%%%%%%%%%%%%%%%%%%%%%%%%%%%%%%%%%%%%%%%%%%%%%%%%%%%%%%%%%%%%%%%%%%%%%%%
%%%%%%%%%%%%%%%%%%%%%%%%%%%%%%%%%%%%%%%%%%%%%%%%%%%%%%%%%%%%%%%%%%%%%%%%%%%%%%%%%%%%%%%%%%

\begin{abstract}
\noindent We study the subgroup $\mathcal{E}xt^i_{\mathcal{C}}(\mathcal{F}; C, D)$ of $\mathcal{E}xt^i_{\mathcal{C}}(C,D)$ formed by those $i$-extensions of $C$ by $D$ in an Abelian category $\mathcal{C}$ which are ${\rm Hom}_{\mathcal{C}}(\mathcal{F},-)$-exact, and present a Baer-like description of this subgroup in terms of certain right derived functors of ${\rm Hom}_{\mathcal{C}}(-,-)$. We also study adjointness properties of these subgroups and the disk and sphere chain complex functors $\mathcal{C} \longrightarrow \Complexes$, given by a collection of natural isomorphisms which generalize the corresponding adjointness properties proven by J. Gillespie for $\mathcal{E}xt^i(-,-)$.  
\end{abstract}  

%\textit{\textbf{Keywords:}} Relative extensions, derived functors, disk complexes, sphere complexes. 

%\textbf{AMS 2010 \textit{Mathematics Subject Classification:}} 18 Category Theory; Homological Algebra.

\tableofcontents

%%%%%%%%%%%%%%%%%%%%%%%%%%%%%%%%%%%%%%%%%%%%%%%%%%%%%%%%%%%%%%%%%%%%%%%%%%%%%%%%%%%%%%%%%%
%%%%%%%%%%%%%%%%%%%%%%%%%%%%%%%%%%%%%%%%%%%%%%%%%%%%%%%%%%%%%%%%%%%%%%%%%%%%%%%%%%%%%%%%%%

\section{Introduction}

Let $\mathcal{C}$ be an Abelian category and $\Complexes$ denote the category of chain complexes over $\mathcal{C}$. Given a chain complex $X$ with differential maps $\partial^X_m : X_m \longrightarrow X_{m-1}$, for each $m \in \mathbb{Z}$ we consider three objects associated to $m$, namely: $X_m$, $Z_m(X) = {\rm Ker}(\partial^X_m)$, and $X_m / B_m(X)$, where $B_m(X) = {\rm Im}(\partial^X_{m+1})$. These particular choices of objects are functorial, i.e. they define the following functors $\Complexes \longrightarrow \mathcal{C}$:
\begin{itemize}[noitemsep, topsep=-10pt]
\item The \underline{$m$-component functor $(-)_m : \Complexes \longrightarrow \mathcal{C}$} is given by $X \mapsto X_m$ for every complex $X$, and if $f : X \longrightarrow Y$ is a chain map, then $f$ is mapped to the morphism $f_m : X_m \longrightarrow Y_m$ in $\mathcal{C}$.

\item The \underline{$m$-cycle functor $Z_m : \Complexes \longrightarrow \mathcal{C}$} is given by $X \mapsto Z_m(X)$ for every complex $X$, and if ${f : X \longrightarrow Y}$ is a chain map, then $Z_m(f)$ is the only morphism $Z_m(X) \longrightarrow Z_m(Y)$ induced by the universal property of kernels. 

\item The \underline{$m$-quotient functor $Q_m : \Complexes \longrightarrow \mathcal{C}$} is given by $X \mapsto X_m / B_m(X)$ for every complex $X$, and if $f : X \longrightarrow Y$ is a chain map, then $Q_m(f)$ is the only morphism $X_m / B_m(X) \longrightarrow Y_m / B_m(Y)$ induced by the universal property of cokernels. 
\end{itemize}

On the other hand, for every object $C$ in $\mathcal{C}$ and for every integer $m \in \mathbb{Z}$, there are two chain complexes associated to $C$:
\begin{itemize}[noitemsep, topsep=-10pt]
\item The \underline{$m$-disk complex centred at $C$}, $D^m(C)$ is defined to be $C$ in degrees $m$ and $m-1$, and zero in all other degrees,  whose differential maps are all zero except for $\partial^{D^m(C)}_m = {\rm id}_C$. 

\item The \underline{$m$-sphere complex centred at $C$}, $S^m(C)$ is defined to be $C$ in degree $m$, and zero in all other degrees, whose differential maps are all zero. 
\end{itemize}

Disk and sphere complexes define functors $D^m, S^m : \mathcal{C} \longrightarrow \Complexes$. It is not hard to see that $D^m$ is a left adjoint of $(-)_m$, and a right adjoint of $(-)_{m-1}$. On the other hand, $S^m$ is a left adjoint of $Z_m$ and a right adjoint of $Q_m$. This can be restated as follows. \\

\begin{proposition}[See {\rm \cite[Lemma 3.1, (1), (2), (3) \& (4)]{Gil}}] If $C$ is an object of $\mathcal{C}$ and $X$ and $Y$ are chain complexes over $\mathcal{C}$, we have the following natural isomorphisms:
\begin{itemize}[noitemsep, topsep=-10pt]
\item[{\bf (1)}] ${\rm Hom}_{\mathcal{C}}(X_{m-1},C) \cong {\rm Hom}_{\Complexes}(X, D^m(C))$.
\item[{\bf (2)}] ${\rm Hom}_{\mathcal{C}}(C, Y_m) \cong {\rm Hom}_{\Complexes}(D^m(C), Y)$.
\item[{\bf (3)}] ${\rm Hom}_{\mathcal{C}}(X_m / B_m(X), C) \cong {\rm Hom}_{\Complexes}(X, S^m(C))$.
\item[{\bf (4)}] ${\rm Hom}_{\mathcal{C}}(C, Z_m(Y)) \cong {\rm Hom}_{\Complexes}(S^m(C),Y)$.
\end{itemize}
\end{proposition}

In the case $\mathcal{C}$ is equipped with enough projective and injective objects, we can compute the extension functors ${\rm Ext}^i_{\mathcal{C}}(-,-)$. Recall that ${\rm Hom}_{\mathcal{C}}(-,-) = {\rm Ext}^0_{\mathcal{C}}(-,-)$. The previous adjointness relations are also valid for $i > 0$, under certain hypothesis. In 2004, J. Gillespie proved in \cite[Lemma 3.1, (5) \& (6)]{Gil} that ${\rm Ext}^1_{\mathcal{C}}(X_{m-1},C) \cong {\rm Ext}^1_{\Complexes}(X, D^m(C))$ and ${\rm Ext}^1_{\mathcal{C}}(C, Y_m) \cong {\rm Ext}^1_{\Complexes}(D^m(C), Y)$. Four years later, the same author proved in \cite[Lemma 4.2]{Gillespie} that the remaining isomorphisms ${\rm Ext}^1_{\mathcal{C}}(X_m / B_m(X), C) \cong {\rm Ext}^1_{\Complexes}(X, S^m(C))$ and ${\rm Ext}^1_{\mathcal{C}}(C, Z_m(Y)) \cong {\rm Ext}^1_{\Complexes}(S^m(C),Y)$ also hold in the case $X$ and $Y$ are exact. These isomorphisms have become an important tool in the study of cotorsion pairs of chain complexes and modules. Since cotorsion pairs are, generally speaking, defined by two classes of objects (say modules or complexes over them) orthogonal to each other with respect to ${\rm Ext}^1(-,-)$, in some cases checking that two complexes are orthogonal reduces to verify the orthogonality between their corresponding terms, cycles or quotients by boundaries. 

The construction of Gillespie's isomorphisms are based on the Baer description of extension functors. Recall that if $\mathcal{C}$ is an Abelian category equipped with either enough projective or injective objects, then ${\rm Ext}^1_{\mathcal{C}}(C,D)$ can be described as the group of classes of extensions of $C$ by $D$, i.e. short exact sequences of the form $S = 0 \longrightarrow D \longrightarrow Z \longrightarrow C \longrightarrow 0$, under a certain equivalence relation. We shall denote this group by $\mathcal{E}xt^1_{\mathcal{C}}(C,D)$. 

The goal of this paper is to study Gillespie's adjointness properties in the context of relative homological algebra. For this purpose, it is useful to consider certain subgroups of $\mathcal{E}xt^1_{\mathcal{C}}(C,D)$. Namely, if $\mathcal{F}$ is a class of objects of $\mathcal{C}$, then we denote $\mathcal{E}^1_{\mathcal{C}}(\mathcal{F}; C, D)$ the subgroup formed by those classes of extensions $S$ which are also exact \textquotedblleft relative to\textquotedblright \ $\mathcal{F}$, i.e. that ${\rm Hom}_{\mathcal{C}}(F,S)$ is exact for every $F \in \mathcal{F}$. One interesting fact we shall prove about these subgroups is that if $\mathcal{F}$ is an special pre-covering class, then $\mathcal{E}xt^1_{\mathcal{C}}(\mathcal{F}; C, D)$ is isomorphic to ${\rm Ext}^1_{\mathcal{C}}(\mathcal{F}; C,D)$, the first right derived functor of ${\rm Hom}_{\mathcal{C}}(-,-)$ computed by using resolutions of $C$ by objects in $\mathcal{F}$.

The contents of these paper are organized as follows. In Section 2 we recall the notions of pre-covering and pre-enveloping classes, left and right resolutions, and how they are used to obtain right derived functors of ${\rm Hom}_{\mathcal{C}}(-,-)$. Then in Section 3 we study the subgroups $\mathcal{E}xt^1_{\mathcal{C}}(\mathcal{F};C,D)$ of relative extensions and construct an isomorphism onto the right derived functors ${\rm Ext}^1_{\mathcal{C}}(\mathcal{F};C,D)$ in the particular case where $\mathcal{F}$ is a special pre-covering class. Section 4 is devoted to extend Gillespie's adjointness properties to the context of relative extensions. We recall the classes $\widetilde{\mathcal{F}}$ and ${\rm dw}\widetilde{\mathcal{F}}$ of $\mathcal{F}$-complexes and degreewise $\mathcal{F}$-complexes induced by a class $\mathcal{F}$ of objects in $\mathcal{C}$. We prove that the groups ${\rm Ext}^1_{\mathcal{C}}(\mathcal{F}; X_{m-1},C)$ and ${\rm Ext}^1_{\Complexes}({\rm dw}\widetilde{\mathcal{F}}; X, D^m(C))$ are isomorphic, and that the same is true for $\mathcal{F}$. Later in Section 5 we continue our study of relative extensions applied to sphere chain complexes. We show that there are natural monomorphisms $\mathcal{E}xt^i_{\mathcal{C}}(\mathcal{F}; \frac{X_m}{B_m(X)}, C) \hookrightarrow \mathcal{E}xt^i_{\Complexes}(\widetilde{\mathcal{F}}; X, S^{m}(C))$ and $\mathcal{E}xt^i_{\mathcal{C}}(\mathcal{F}; C, Z_m(Y)) \hookrightarrow \mathcal{E}xt^i_{\Complexes}(\widetilde{\mathcal{F}}; S^m(C), Y)$, which are actually isomorphisms in the case where $X$ and $Y$ are exact and ${\rm Hom}_{\mathcal{C}}(\mathcal{F},-)$-exact. We conclude this work presenting some applications of our results in the context of Gorenstein homological algebra, where we shall work with modules and chain complexes over a Gorenstein ring. In this particular setting, Gorenstein-extension functors ${\rm GExt}^i(-,-)$ have their Baer description with respect to the class of Gorenstein-projective modules (or complexes), since these modules from a special pre-covering class. Moreover, it turns out that the class of Gorenstein-projective complexes coincides with the class of differential graded Gorenstien-projective complexes, and we shall use this characterization to provide another proof of the adjointness properties of ${\rm GExt}^i(-,-)$ and sphere chain complexes. 

%These functors play an important role in the study of balance of ${\rm Hom}_R(-,-)$ over Gorenstein rings, and they can 

%%%%%%%%%%%%%%%%%%%%%%%%%%%%%%%%%%%%%%%%%%%%%%%%%%%%%%%%%%%%%%%%%%%%%%%%%%%%%%%%%%%%%%%%%%
%%%%%%%%%%%%%%%%%%%%%%%%%%%%%%%%%%%%%%%%%%%%%%%%%%%%%%%%%%%%%%%%%%%%%%%%%%%%%%%%%%%%%%%%%%

\section{Pre-covering classes and right derived functors of ${\rm Hom}_{\mathcal{C}}(-,-)$}

In this section we recall the notion of derived functors, as one of the key concepts in this work. We focus in the particular case of getting right derived functors of ${\rm Hom}_{\mathcal{C}}(-,-)$ from resolutions by a certain class of objects in an Abelian category. The theoretic setting presented below includes the computation of extension ${\rm Ext}^i(-,-)$ and Gorenstein-extension ${\rm GExt}^i(-,-)$ functors. If the reader in interested in more details on these topics, a good reference is \cite[Chapters 8 \& 12]{EJ}. \\

\begin{definition} Let $\mathcal{F}$ be a class of objects in an Abelian category $\mathcal{C}$.
\begin{itemize}[noitemsep, topsep=-10pt]
\item[{\bf (1)}] {\rm \cite[Definition 8.1.1]{EJ}} A chain complex $X = \cdots \longrightarrow X_{m+1} \longrightarrow X_m \longrightarrow X_{m-1} \longrightarrow \cdots$ is said to be \underline{${\rm Hom}_{\mathcal{C}}(\mathcal{F}, -)$-exact} if for every object $F$ of $\mathcal{F}$, the complex of Abelian groups \[ {\rm Hom}_{\mathcal{C}}(F,X) = \cdots \longrightarrow {\rm Hom}_{\mathcal{C}}(F, X_{m+1}) \longrightarrow {\rm Hom}_{\mathcal{C}}(F, X_m) \longrightarrow {\rm Hom}_{\mathcal{C}}(F, X_{m-1}) \longrightarrow \cdots \] is exact. The notion of ${\rm Hom}_{\mathcal{C}}(-, \mathcal{F})$-exact complex is dual. 

\item[{\bf (2)}] {\rm \cite[Definition 8.1.2]{EJ}} A \underline{left $\mathcal{F}$-resolution} of an object $C$ of $\mathcal{C}$ is a ${\rm Hom}_{\mathcal{C}}(\mathcal{F},-)$-exact (but not necessarily exact) complex $\cdots \longrightarrow F_1 \longrightarrow F_0 \longrightarrow C \longrightarrow 0$ where $F_m \in \mathcal{F}$ for every $m \geq 0$. Right $\mathcal{F}$-resolutions are defined dually. 

\item[{\bf (3)}] {\rm \cite[Definition 5.1.1]{EJ}} A morphism $f : F \longrightarrow C$ with $F \in \mathcal{F}$ is said to be an \underline{$\mathcal{F}$-cover} of $C$ if: 
\begin{itemize}[noitemsep, topsep=10pt]
\item[ {\bf (i)}] Given another morphism $f' : F' \longrightarrow C$ with $F' \in \mathcal{F}$, there exists a morphism $\varphi : F' \longrightarrow F$ (not necessarily unique) such that $f' = f \circ \varphi$. 

\item[ {\bf (ii)}] If $F' = F$ then $\varphi$ is an automorphism of $F$. 
\end{itemize}
If $f$ satisfies {\bf (i)} but may be not {\bf (ii)}, then it is called an \underline{$\mathcal{F}$-pre-cover}. The class $\mathcal{F}$ is called a \underline{(pre-)covering class} if every object of $\mathcal{C}$ has an $\mathcal{F}$-(pre-)cover. The dual notions of $\mathcal{F}$-covers and $\mathcal{F}$-pre-covers are those of $\mathcal{F}$-envelopes and $\mathcal{F}$-pre-envelopes.
\end{itemize}
\end{definition}

The following proposition is not hard to prove. \\

\begin{proposition} If $\mathcal{F}$ is a pre-covering class in $\mathcal{C}$, then every object of $\mathcal{C}$ has a left $\mathcal{F}$-resolution. Dually, if $\mathcal{F}$ is a pre-enveloping class in $\mathcal{C}$, then every object of $\mathcal{C}$ has a right $\mathcal{F}$-resolution. 
\end{proposition}

Let $T : \mathcal{C} \longrightarrow \mathcal{D}$ be a covariant functor between Abelian categories. Let $\mathcal{G}$ be a pre-enveloping class of $\mathcal{C}$ and $C$ an object in $\mathcal{C}$. Consider a right $\mathcal{G}$-resolution $0 \longrightarrow C \longrightarrow G^0 \longrightarrow G^1 \longrightarrow \cdots$ of $C$, which exists by the previous proposition, and denote by $\textbf{\textit{G}}^{\bullet} = G^0 \longrightarrow G^1 \longrightarrow \cdots$ the complex obtained after deleting the term $C$. The cohomology of the complex $T(\textbf{\textit{G}}^{\bullet})$ defines the \underline{right derived functors} of $T$, denoted $R^i T : C \mapsto (R^i T)(C)$. If $T$ is contravariant, then the right derived functors can be computed using left $\mathcal{F}$-resolutions of $C$. \newpage

\begin{example} Let $C$ and $D$ be two objects of $\mathcal{C}$, and $\mathcal{F}$ and $\mathcal{G}$ as above. The right $i$th derived functor of ${\rm Hom}_{\mathcal{C}}(-,D)$ evaluated at $C$ is defined as the $i$th cohomology of ${\rm Hom}_{\mathcal{C}}(\textbf{\textit{F}}_\bullet, D)$, and is denoted by \[ {\rm Ext}^i_{\mathcal{C}}(\mathcal{F}; C,D) := R^i({\rm Hom}_{\mathcal{C}}(-,D))(C). \] Dually, the $i$th cohomology of the complex ${\rm Hom}_{\mathcal{C}}(C,\textbf{\textit{G}}_\bullet)$ defines the right $i$th derived functor of ${\rm Hom}_{\mathcal{C}}(C,-)$ evaluated at $D$, denoted by \[ {\rm Ext}^i_{\mathcal{C}}(C,D; \mathcal{G}) := R^i({\rm Hom}_{\mathcal{C}}(C,-))(D). \] 
In the case where $\mathcal{F} = \mathcal{P}_{roj}(\mathcal{C})$ is the class of projective objects of an Abelian category $\mathcal{C}$ with enough projective objects (so $\mathcal{P}_{roj}(\mathcal{C})$ is pre-covering), then ${\rm Ext}^i_{\mathcal{C}}(\mathcal{P}_{roj}(\mathcal{C}); C, D)$ is the standard $i$th extension ${\rm Ext}^i_{\mathcal{C}}(C,D)$ (Notice that we may choose an exact (left) projective resolution of $C$). Moreover, if $\mathcal{I}_{nj}(\mathcal{C})$ denotes the class of injective objects, the groups ${\rm Ext}^i_{\mathcal{C}}(\mathcal{P}_{roj}(\mathcal{C}); C, D)$ and ${\rm Ext}^i_{\mathcal{C}}(C, D; \mathcal{I}_{nj}(\mathcal{C}))$ coincide when $\mathcal{C}$ has enough projective and injective objects. 

Another interesting case is when we put $\mathcal{F} = \mathcal{GP}_{roj}$ and $\mathcal{G} = \mathcal{GI}_{nj}$ as the classes of Gorenstein-projective and Gorenstein-injective modules, respectively. In the particular setting when $R$ is a Gorenstein ring, we can compute (exact) left Gorenstein-projective and right Gorenstein-injective resolutions of every module, and the groups ${\rm Ext}^i_{R}(\mathcal{GP}_{roj}; C, D)$ and ${\rm Ext}^i_R(C, D; \mathcal{GI}_{nj})$ coincide for every pair of modules $C$ and $D$. There will be more to be said about these classes in Section 6. 
\end{example}

%%%%%%%%%%%%%%%%%%%%%%%%%%%%%%%%%%%%%%%%%%%%%%%%%%%%%%%%%%%%%%%%%%%%%%%%%%%%%%%%%%%%%%%%%%
%%%%%%%%%%%%%%%%%%%%%%%%%%%%%%%%%%%%%%%%%%%%%%%%%%%%%%%%%%%%%%%%%%%%%%%%%%%%%%%%%%%%%%%%%%

\section{Baer description of $\mathcal{F}$-extension functors}

Given two objects $C$ and $D$ in an Abelian category $\mathcal{C}$, by an \underline{$i$-extension of $C$ by $D$} be mean an exact sequence of the form $S = 0 \longrightarrow D \longrightarrow E^i \longrightarrow \cdots \longrightarrow E^1 \longrightarrow C \longrightarrow 0$. We say that two exact sequences 
\begin{align*}
S & = 0 \longrightarrow D \longrightarrow E^i \longrightarrow \cdots \longrightarrow E^1 \longrightarrow C \longrightarrow 0 \mbox{ and } \hat{S} = 0 \longrightarrow D \longrightarrow \hat{E}^i \longrightarrow \cdots \longrightarrow \hat{E}^1 \longrightarrow C \longrightarrow 0
\end{align*}
are \underline{related} (denoted $S \sim \hat{S}$) if there exist morphisms $E^k \longrightarrow \hat{E}^k$ for every $1 \leq k \leq i$ such that the diagram
\[ \begin{tikzpicture}
\matrix (m) [matrix of math nodes, row sep=2em, column sep=2em, text height=1.5ex, text depth=0.25ex]
{ 0 & D & E^i & \cdots & E^1 & C & 0 \\ 0 & D & \hat{E}^i & \cdots & \hat{E}^1 & C & 0 \\ };
\path[->]
(m-1-1) edge (m-1-2) (m-1-2) edge (m-1-3) (m-1-3) edge (m-1-4) edge (m-2-3) (m-1-4) edge (m-1-5) (m-1-5) edge (m-1-6) edge (m-2-5) (m-1-6) edge (m-1-7)
(m-2-1) edge (m-2-2) (m-2-2) edge (m-2-3) (m-2-3) edge (m-2-4) (m-2-4) edge (m-2-5) (m-2-5) edge (m-2-6) (m-2-6) edge (m-2-7);
\path[-,font=\scriptsize]
(m-1-2) edge [double, thick, double distance=4pt] (m-2-2)
(m-1-6) edge [double, thick, double distance=4pt] (m-2-6);
\end{tikzpicture} \]
commutes. We shall denote by $\mathcal{E}{xt}^i_{\mathcal{C}}(C,D)$ the set of classes of $i$-extensions under the equivalence relation generated by $\sim$. \\

\begin{remark} Note that in the case $i = 1$, the equivalence relation generated by $\sim$ is $\sim$ itself, since the arrow $E^1 \longrightarrow \hat{E}^1$ is an isomorphism. 
\end{remark}

The set $\mathcal{E}{xt}^i_{\mathcal{C}}(C,D)$ has an Abelian group structure, given by a binary operation known as the \underline{Baer sum}. 

Suppose we are given two classes $[S_1]$ and $[S_2]$, where 
\begin{align*}
S_1 & = 0 \longrightarrow D \longrightarrow E^i_1 \longrightarrow \cdots \longrightarrow E^1_1 \longrightarrow C \longrightarrow 0 \mbox{ and } S_2 = 0 \longrightarrow D \longrightarrow E^i_2 \longrightarrow \cdots \longrightarrow E^1_2 \longrightarrow C \longrightarrow 0.
\end{align*}
The Baer sum $[S_1] +_B [S_2]$ of $[S_1]$ and $[S_2]$ is defined by the following steps:  
\begin{itemize}[noitemsep, topsep=-10pt]
\item[{\bf (1)}] Take the direct sum of $S_1$ and $S_2$, \[ S_1 \oplus S_2 = 0 \longrightarrow D \oplus D \longrightarrow E^i_1 \oplus E^i_2 \longrightarrow \cdots \longrightarrow E^1_1 \oplus E^1_2 \longrightarrow C \oplus C \longrightarrow 0. \] 

\item[{\bf (2)}] After taking the pullback of $\Delta_C : C \longrightarrow C \oplus C$ and $E^1_1 \oplus E^1_2 \longrightarrow C \oplus C$, we get a commutative diagram 
\[ \begin{tikzpicture}
\matrix (m) [matrix of math nodes, row sep=2em, column sep=2em, text height=1.5ex, text depth=0.25ex]
{ 0 & D \oplus D & E^i_1 \oplus E^i_2 & \cdots & E^2_1 \oplus E^2_2 & (E^1_1 \oplus E^1_2) \times_{C \oplus C} C & C & 0 \\ 0 & D \oplus D & E^i_1 \oplus E^i_2 & \cdots & E^2_1 \oplus E^2_2 & E^1_1 \oplus E^1_2 & C \oplus C & 0 \\ };
\path[->]
(m-1-1) edge (m-1-2) (m-1-2) edge (m-1-3) (m-1-3) edge (m-1-4) (m-1-4) edge (m-1-5) (m-1-5) edge (m-1-6) (m-1-6) edge (m-1-7) edge (m-2-6) (m-1-7) edge (m-1-8) edge node[right] {$\Delta_C$} (m-2-7)
(m-2-1) edge (m-2-2) (m-2-2) edge (m-2-3) (m-2-3) edge (m-2-4) (m-2-4) edge (m-2-5) (m-2-5) edge (m-2-6) (m-2-6) edge (m-2-7) (m-2-7) edge (m-2-8);
\path[-,font=\scriptsize]
(m-1-2) edge [double, thick, double distance=4pt] (m-2-2)
(m-1-3) edge [double, thick, double distance=4pt] (m-2-3)
(m-1-5) edge [double, thick, double distance=4pt] (m-2-5);
\end{tikzpicture} \]

\item[{\bf (3)}] Finally, take the pushout of $\nabla_D : D \oplus D \longrightarrow D$ and $C \oplus C \longrightarrow E^i_1 \oplus E^i_2$, and get a commutative diagram 
\[ \begin{tikzpicture}
\matrix (m) [matrix of math nodes, row sep=2em, column sep=2em, text height=1.5ex, text depth=0.25ex]
{ 0 & D \oplus D & E^i_1 \oplus E^i_2 & \cdots & (E^1_1 \oplus E^1_2) \times_{C \oplus C} C & C & 0 \\ 0 & D & D \coprod_{D \oplus D} (E^i_1 \oplus E^i_2) & \cdots & (E^1_1 \oplus E^1_2) \times_{C \oplus C} C & C & 0  \\ };
\path[->]
(m-1-1) edge (m-1-2) (m-1-2) edge (m-1-3) edge node[right] {$\nabla_D$} (m-2-2) (m-1-3) edge (m-1-4) edge (m-2-3) (m-1-4) edge (m-1-5) (m-1-5) edge (m-1-6) edge (m-2-5) (m-1-6) edge (m-1-7) edge (m-2-6) 
(m-2-1) edge (m-2-2) (m-2-2) edge (m-2-3) (m-2-3) edge (m-2-4) (m-2-4) edge (m-2-5) (m-2-5) edge (m-2-6) (m-2-6) edge (m-2-7);
\path[-,font=\scriptsize]
(m-1-6) edge [double, thick, double distance=4pt] (m-2-6);
\end{tikzpicture} \] 
\end{itemize}
The Baer sum $[S_1] +_B [S_2]$ is given by the class of the bottom row in the diagram above. 

The importance of the groups $\mathcal{E}{xt}^i_{\mathcal{C}}(C,D)$ lies in the fact that they can be used to describe the extension functors ${\rm Ext}^i_{\mathcal{C}}(C,D)$. \\

\begin{proposition} If $\mathcal{C}$ is an Abelian category with either enough projective or injective objects, then the groups $\mathcal{E}{xt}^i_{\mathcal{C}}(C,D)$ and ${\rm Ext}^i_{\mathcal{C}}(C,D)$ are isomorphic. 
\end{proposition}

We skip the proof of this (well known) result, since we shall provide a generalization in the next lines. This generalization consists in giving a Baer-like description of ${\rm Ext}^i_{\mathcal{C}}(\mathcal{F}; C, D)$ and ${\rm Ext}^i_{\mathcal{C}}(C, D; \mathcal{G})$, by constructing isomorphisms from them to certain subgroups of $\mathcal{E}{xt}^i_{\mathcal{C}}(C,D)$. \\

\begin{definition} Let $\mathcal{F}$ be a class of objects of an Abelian category $\mathcal{C}$. We shall say that an $i$-extension of $C$ by $D$ is \underline{left-relative to $\mathcal{F}$} if it is ${\rm Hom}_{\mathcal{C}}(\mathcal{F},-)$-exact as a chain complex. Extensions \underline{right-relative to $\mathcal{F}$} are defined dually. We shall denote by $\mathcal{E}{xt}^i_{\mathcal{C}}(\mathcal{F}; C, D)$ (resp. $\mathcal{E}{xt}^i_{\mathcal{C}}(C, D; \mathcal{F})$) the subset of $\mathcal{E}{xt}^i_{\mathcal{C}}(C, D)$ formed by the classes of $i$-extensions of $C$ by $D$ which are left-relative (resp. right-relative) to $\mathcal{F}$. 
\end{definition}

\begin{proposition} $\mathcal{E}{xt}^i_{\mathcal{C}}(\mathcal{F}; C, D)$ and $\mathcal{E}{xt}^i_{\mathcal{C}}(C, D; \mathcal{F})$ are sub-groups of $\mathcal{E}{xt}^i_{\mathcal{C}}(C,D)$. 
\end{proposition}
\begin{proof} We only prove that $\mathcal{E}{xt}^i_{\mathcal{C}}(\mathcal{F}; C, D)$ is a sub-group of $\mathcal{E}{xt}^i_{\mathcal{C}}(C,D)$ for the case $i = 1$. First, note that $\mathcal{E}{xt}^1_{\mathcal{C}}(\mathcal{F}; C, D)$ is nonempty since the representative $0 \longrightarrow D \longrightarrow C \oplus D \longrightarrow C \longrightarrow 0$ of the zero element is left-relative to $\mathcal{F}$. Now suppose we are given two extensions of $C$ by $D$ left-relative to $\mathcal{F}$, say $S_1 = (0 \longrightarrow D \longrightarrow E_1 \longrightarrow C \longrightarrow 0)$ and $S_2 = (0 \longrightarrow D \longrightarrow E_2 \longrightarrow C \longrightarrow 0)$. We show that the sequence \[ 0 \longrightarrow {\rm Hom}_{\mathcal{C}}(F, D) \longrightarrow {\rm Hom}_{\mathcal{C}}\left(F, D \coprod_{D \oplus D} \left[ (E_1 \oplus E_2) \times_{C \oplus C} C \right]\right) \longrightarrow {\rm Hom}_{\mathcal{C}}(F, C) \longrightarrow 0 \] is exact for every $F \in \mathcal{F}$. 
\begin{itemize}[noitemsep, topsep=-10pt]
\item[{\bf (1)}] Note that the sequence ${\rm Hom}_{\mathcal{C}}(F, S_1 \oplus S_2)$ is exact since is it isomorphic to the direct sum of ${\rm Hom}_{\mathcal{C}}(F, S_1)$ and ${\rm Hom}_{\mathcal{C}}(F, S_2)$, which are exact. 

\item[{\bf (2)}] To prove $0 \longrightarrow {\rm Hom}_{\mathcal{C}}(F, D \oplus D) \longrightarrow {\rm Hom}_{\mathcal{C}}\left(F, (E_1 \oplus E_2) \times_{C \oplus C} C \right) \longrightarrow {\rm Hom}_{\mathcal{C}}(F, C) \longrightarrow 0$ is exact, it suffices to show that the morphism ${\rm Hom}_{\mathcal{C}}\left(F, (E_1 \oplus E_2) \times_{C \oplus C} C \right) \longrightarrow {\rm Hom}_{\mathcal{C}}(F, C)$ is surjective, since the functor ${\rm Hom}_{\mathcal{C}}(F, -)$ is left exact. Suppose we are given a morphism $f : F \longrightarrow C$. Then $\Delta_C \circ f \in {\rm Hom}_{\mathcal{C}}(F, C \oplus C)$. Since ${\rm Hom}_{\mathcal{C}}(F, S_1 \oplus S_2)$ is exact, there exists a morphism $g : F \longrightarrow E_1 \oplus E_2$ such that $\Delta_C \circ f = (\beta_1 \oplus \beta_2) \circ g$. It follows by the universal property of pullbacks that there exists a unique morphism $h : F \longrightarrow (E_1 \oplus E_2) \times_{C \oplus C} C$ such that the following diagram commutes:
\[ \begin{tikzpicture}
\matrix (m) [matrix of math nodes, row sep=2em, column sep=2em, text height=1.5ex, text depth=0.25ex]
{ F \\ & (E_1 \oplus E_2) \times_{C \oplus C} C & C \\ & E_1 \oplus E_2 & C \oplus C \\ };
\path[->]
(m-1-1) edge [bend left = 30] node[above,sloped] {$f$} (m-2-3) edge [bend right=30] node[below,sloped] {$g$} (m-3-2)
(m-2-2) edge (m-2-3) edge (m-3-2)
(m-2-3) edge node[right] {$\Delta_C$} (m-3-3)
(m-3-2) edge node[below] {$\beta_1 \oplus \beta_2$} (m-3-3);
\path[dotted,->]
(m-1-1) edge node[above,sloped] {$\exists \mbox{! } h$} (m-2-2);
\end{tikzpicture} \]
Hence, $f = {\rm Hom}_{\mathcal{C}}(F, (E_1 \oplus E_2) \times_{C \oplus C} C \longrightarrow C)(h)$. 

\item[{\bf (3)}] Finally, we show that the morphism ${\rm Hom}_{\mathcal{C}}(F, D \coprod_{D \oplus D} \left[ (E_1 \oplus E_2) \times_{C \oplus C} C \right]) \longrightarrow {\rm Hom}_{\mathcal{C}}(F, C)$ is surjective. We have the following commutative diagram
\[ \begin{tikzpicture}
\matrix (m) [matrix of math nodes, row sep=2em, column sep=2em, text height=1.5ex, text depth=0.25ex]
{ 0 & {\rm Hom}_{\mathcal{C}}(F, D \oplus D) & {\rm Hom}_{\mathcal{C}}(F, (E_1 \oplus E_2) \times_{C \oplus C} C) & {\rm Hom}_{\mathcal{C}}(F, C) & 0 \\ 0 & {\rm Hom}_{\mathcal{C}}(F, D) & {\rm Hom}_{\mathcal{C}}(F, D \coprod_{D \oplus D} \left[ (E_1 \oplus E_2) \times_{C \oplus C} C \right]) & {\rm Hom}_{\mathcal{C}}(F, C) & 0 \\ };
\path[->]
(m-1-1) edge (m-1-2) (m-1-2) edge (m-1-3) edge (m-2-2) (m-1-3) edge (m-1-4) edge (m-2-3) (m-1-4) edge (m-1-5)
(m-2-1) edge (m-2-2) (m-2-2) edge (m-2-3) (m-2-3) edge (m-2-4) (m-2-4) edge (m-2-5);
\path[-,font=\scriptsize]
(m-1-4) edge [double, thick, double distance=4pt] (m-2-4);
\end{tikzpicture} \]
where the top row is exact. Using diagram chasing, it is not hard to show that the bottom row is also exact. 
\end{itemize}
Therefore, $[S_1] +_B [S_2] \in \mathcal{E}{xt}^1_{\mathcal{C}}(\mathcal{F}; C, D)$. 
\end{proof}

Now we focus on proving that ${\rm Ext}^i_{\mathcal{C}}(\mathcal{F}; C,D)$ is isomorphic to $\mathcal{E}{xt}^i_{\mathcal{C}}(\mathcal{F}; C, D)$. As a first approach, it is well known that an isomorphism between ${\rm Ext}^i_{\mathcal{C}}(C,D)$ and $\mathcal{E}{xt}^i_{\mathcal{C}}(C,D)$ can be constructed by using an exact (left) projective resolution of $C$ (This is possible in Abelian categories with enough projective objects). So we may think of considering left $\mathcal{F}$-resolutions of $C$ to get a map from $\mathcal{E}{xt}^i_{\mathcal{C}}(\mathcal{F}; C, D)$ to ${\rm Ext}^i_{\mathcal{C}}(\mathcal{F}; C,D)$. However, left $\mathcal{F}$-resolutions need not be exact. This limitation can be avoided if we impose an extra condition on $\mathcal{F}$, related to a special type of pre-covering classes. \\

\begin{definition} Let $\mathcal{F}$ be a class of objects in an Abelian category $\mathcal{C}$. 
\begin{itemize}[noitemsep, topsep=-10pt]
\item[{\bf (1)}] The \underline{left orthogonal class} of $\mathcal{F}$ is defined as $\mathcal{F}^\perp := \{ D \in {\rm Ob}(\mathcal{C}) \mbox{ $:$ } {\rm Ext}^1_{\mathcal{C}}(F, D) = 0, \mbox{ $\forall$ }F \in \mathcal{F} \}$. 

\item[{\bf (2)}] {\rm \cite[Definition 7.1.6]{EJ}} A morphism $F \longrightarrow C$, with $F \in \mathcal{F}$, is a special $\mathcal{F}$-pre-cover of $C$ if it is an epimorphism and if ${\rm Ker}(F \longrightarrow C) \in \mathcal{F}^\perp$.  

\item[{\bf (3)}] The class $\mathcal{F}$ is said to be a \underline{special pre-covering class} if every object has a special $\mathcal{F}$-pre-cover.
\end{itemize}
The concepts of \underline{right orthogonal class}, \underline{special pre-envelope} and \underline{special pre-enveloping class} are dual. 
\end{definition}

Note that every special pre-covering (resp. special pre-enveloping) class is a pre-covering class (resp. pre-enveloping class). The following lemma is easy to prove. \\

\begin{lemma} Let $\mathcal{F}$ be a special pre-covering class. Then every object of $\mathcal{C}$ has an exact left $\mathcal{F}$-resolution. \\
\end{lemma}

\begin{theorem}\label{isosext} If $\mathcal{F}$ is a special pre-covering class of objects in an Abelian category $\mathcal{C}$, then there is a group isomorphism between ${\rm Ext}^i_{\mathcal{C}}(\mathcal{F}; C, D)$ and $\mathcal{E}{xt}^i_{\mathcal{C}}(\mathcal{F}; C, D)$, for every pair of objects $C$ and $D$. Dually, if $\mathcal{G}$ is a special pre-enveloping class, then ${\rm Ext}^i_{\mathcal{C}}(C, D; \mathcal{G})$ and $\mathcal{E}{xt}^i_{\mathcal{C}}(C, D; \mathcal{G})$ are isomorphic. 
\end{theorem}
\begin{proof} We only construct an isomorphism between ${\rm Ext}^1_{\mathcal{C}}(\mathcal{F}; C, D)$ and $\mathcal{E}{xt}^1_{\mathcal{C}}(\mathcal{F}; C, D)$. Consider a representative $S = 0 \longrightarrow D \stackrel{\alpha}\longrightarrow E \stackrel{\beta}\longrightarrow C \longrightarrow 0$ of a class in $\mathcal{E}{xt}^1_{\mathcal{C}}(\mathcal{F}; C, D)$. Since $\mathcal{F}$ is special pre-covering, we can obtain an exact left $\mathcal{F}$-resolution $\cdots \longrightarrow F_1 \stackrel{f_1}\longrightarrow F_0 \stackrel{f_0}\longrightarrow C \longrightarrow 0$. Recall ${\rm Ext}^1_{\mathcal{C}}(\mathcal{F}; C, D) = {\rm Ker}({\rm Hom}_{\mathcal{C}}(f_2, D)) / {\rm Im}({\rm Hom}_{\mathcal{C}}(f_1, D))$.  Since $S$ is ${\rm Hom}_{\mathcal{C}}(\mathcal{F}, -)$-exact, the sequence ${\rm Hom}_{\mathcal{C}}(F_0, S)$ is also exact. So there exists a morphism $g_0 : F_0 \longrightarrow E$ such that $f_0 = \beta \circ g_0$. Note that $\beta \circ (g_0 \circ f_1) = 0$, and since $S$ is exact, there exists a unique homomorphism $g_S : F_1 \longrightarrow D$ such that $\alpha \circ g_S = g_0 \circ f_1$. 
\[ \begin{tikzpicture}
\matrix (m) [matrix of math nodes, row sep=2em, column sep=2em, text height=1.5ex, text depth=0.25ex]
{ & & & & 0 \\ & & & & D \\ & & & & E \\ \cdots & F_2 & F_1 & F_0 & C & 0 \\ & & & & 0 \\ };
\path[->]
(m-1-5) edge (m-2-5)
(m-4-5) edge (m-5-5)
(m-4-1) edge (m-4-2) (m-4-2) edge node[below] {$f_2$} (m-4-3) (m-4-3) edge node[below] {$f_1$} (m-4-4) (m-4-4) edge node[below] {$f_0$} (m-4-5) (m-4-5) edge (m-4-6)
(m-4-4) edge node[above,sloped] {$g_0$} (m-3-5) (m-4-3) edge node[above,sloped] {$g_S$} (m-2-5)
(m-2-5) edge node[left] {$\alpha$} (m-3-5)
(m-3-5) edge node[left] {$\beta$} (m-4-5);
\end{tikzpicture} \]
On the other hand, ${\rm Hom}_{\mathcal{C}}(f_2, D)(g_S) = g_S \circ f_2$, and $\alpha \circ (g_S \circ f_2) = g_0 \circ f_1 \circ f_2 = 0$. Since $\alpha$ is a monomorphism, we have $g_S \circ f_2 = 0$. Then $g_S \in {\rm Ker}({\rm Hom}_{\mathcal{C}}(f_2, D))$. One can check that the map 
\begin{align*}
\Phi : \mathcal{E}{xt}^1_{\mathcal{C}}(\mathcal{F}; C, D) & \longrightarrow {\rm Ext}^1_{\mathcal{C}}(\mathcal{F}; C,D) \\ 
[S] & \mapsto g_S + {\rm Im}({\rm Hom}_{\mathcal{C}}(f_1, D))
\end{align*} 
is a well defined group homomorphism, where $g_S + {\rm Im}({\rm Hom}_{\mathcal{C}}(f_1, D))$ is the class of $g_S$ in ${\rm Ext}^1_{\mathcal{C}}(\mathcal{F}; C, D)$.  

Now we show $\Phi$ is monic. Suppose $S = 0 \longrightarrow D \stackrel{\alpha}\longrightarrow E \stackrel{\beta}\longrightarrow C \longrightarrow 0$ is a representative such that $g_S + {\rm Im}({\rm Hom}_{\mathcal{C}}(f_1, D)) = \Phi([S]) = 0 + {\rm Im}({\rm Hom}_{\mathcal{C}}(f_1, D))$. Then $g_S = r \circ f_1$ for some morphism $r : F_0 \longrightarrow D$. It follows $(g_0 - \alpha \circ r) \circ f_1 = 0$ and $\beta \circ (g_0 - \alpha \circ r) = f_0$. Hence we may assume $g_s = 0$.  Note that there is a unique morphism $k_0 : C \longrightarrow E$ such that $k_0 \circ f_0 = g_0$, since $g_0 \circ f_1 = 0$ and the left $\mathcal{F}$-resolution of $C$ is exact. It follows $(\beta \circ k_0) \circ f_0 = f_0$ and so $\beta \circ k_0 = {\rm id}_{C}$, since $f_0$ is epic. 

To show that $\Phi$ is also epic, let $h + {\rm Im}({\rm Hom}_{\mathcal{C}}(f_1, D)) \in {\rm Ext}^1_{\mathcal{C}}(\mathcal{F}; C, D)$. Then we have $h \circ f_2 = 0$, and so there exists a unique morphism $h' : {\rm Ker}(f_0) \longrightarrow D$ such that $h' \circ \widehat{f_1} = h$, where $f_1$ is written as the epic-monic factorization $F_1 \stackrel{\widehat{f_1}}\longrightarrow {\rm Im}(f_1) \stackrel{j_0}\longrightarrow D$. Taking the pushout of $j_0 : {\rm Ker}(f_0) \longrightarrow F_0$ and $h'$, we get the following commutative diagram with exact rows: 
\[ \begin{tikzpicture}
\matrix (m) [matrix of math nodes, row sep=2em, column sep=2em, text height=1.5ex, text depth=0.25ex]
{ 0 & {\rm Ker}(f_0) & F_0 & C & 0 \\ 0 & D & D \coprod_{{\rm Ker}(f_0)} F_0 & C & 0 \\ };
\path[->]
(m-1-1) edge (m-1-2) (m-1-4) edge (m-1-5)
(m-2-1) edge (m-2-2) (m-2-4) edge (m-2-5)
(m-1-2) edge node[above] {$j_0$} (m-1-3)
(m-2-2) edge node[below] {$\alpha$} (m-2-3)
(m-1-3) edge node[above] {$f_0$} (m-1-4)
(m-2-3) edge node[below] {$\beta$} (m-2-4)
(m-1-2) edge node[left] {$h'$} (m-2-2)
(m-1-3) edge node[left] {$i$} (m-2-3);
\path[-,font=\scriptsize]
(m-1-4) edge [double, thick, double distance=4pt] (m-2-4);
\end{tikzpicture} \]
One can check that the following diagram commutes:
\[ \begin{tikzpicture}
\matrix (m) [matrix of math nodes, row sep=2em, column sep=2em, text height=1.5ex, text depth=0.25ex]
{ & & & & & 0 \\ & & & & & D \\ & & & & & D \coprod_{{\rm Ker}(f_0)} F_0 \\ \cdots & F_2 & F_1 & {\rm Ker}(f_0) & F_0 & C & 0 \\ & & & & & 0 \\ };
\path[->]
(m-1-6) edge (m-2-6)
(m-4-6) edge (m-5-6)
(m-2-6) edge node[right] {$\alpha$} (m-3-6)
(m-3-6) edge node[right] {$\beta$} (m-4-6)
(m-4-1) edge (m-4-2) (m-4-2) edge node[above] {$f_2$} (m-4-3) (m-4-3) edge node[above] {$\widehat{f_1}$} (m-4-4) (m-4-4) edge node[above] {$j_0$} (m-4-5) (m-4-5) edge node[above] {$f_0$} (m-4-6) (m-4-6) edge (m-4-7)
(m-4-3) edge [bend left=30] node[above,sloped] {$h' \circ \widehat{f_1}$} (m-2-6)
(m-4-4) edge node[above,sloped] {$h'$} (m-2-6)
(m-4-5) edge node[above,sloped] {$i$} (m-3-6)
(m-4-3) edge [bend right=30] node[below,sloped] {$f_1$} (m-4-5);
\end{tikzpicture} \]
We have the following commutative diagram with exact rows:
\[ \begin{tikzpicture}
\matrix (m) [matrix of math nodes, row sep=2em, column sep=2em, text height=1.5ex, text depth=0.25ex]
{ \cdots & F_2 & F_1 & F_0 & C & 0 \\ & 0 & D & D \coprod_{{\rm Ker}(f_0)} F_0 & C & 0 \\ };
\path[->]
(m-1-1) edge (m-1-2) (m-1-2) edge node[above] {$f_2$} (m-1-3) (m-1-3) edge node[above] {$f_1$} (m-1-4) edge node[left] {$h$} (m-2-3) (m-1-4) edge node[above] {$f_0$} (m-1-5) edge node[left] {$i$} (m-2-4) (m-1-5) edge (m-1-6)
(m-2-2) edge (m-2-3) (m-2-3) edge node[below] {$\alpha$} (m-2-4) (m-2-4) edge node[below] {$\beta$} (m-2-5) (m-2-5) edge (m-2-6);
\path[-,font=\scriptsize]
(m-1-5) edge [double, thick, double distance=4pt] (m-2-5);
\end{tikzpicture} \]
To show that the bottom row is ${\rm Hom}_{\mathcal{C}}(\mathcal{F}, -)$-exact, it suffices to verify that for every $F \in \mathcal{F}$, the homomorphism ${\rm Hom}_{\mathcal{C}}(F, D \coprod_{{\rm Ker}(f_0)} F_0 ) \longrightarrow {\rm Hom}_{\mathcal{C}}(F, C)$ is surjective. The diagram \\
\[ \begin{tikzpicture}
\matrix (m) [matrix of math nodes, row sep=2em, column sep=1.5em, text height=1.5ex, text depth=0.25ex]
{ \cdots & {\rm Hom}_{\mathcal{C}}(F, F_2) & {\rm Hom}_{\mathcal{C}}(F, F_1) & {\rm Hom}_{\mathcal{C}}(F, F_0) & {\rm Hom}_{\mathcal{C}}(F, C) & 0 \\ & 0 & {\rm Hom}_{\mathcal{C}}(F, D) & {\rm Hom}_{\mathcal{C}}(F, D \coprod_{{\rm Ker}(f_0)} F_0) & {\rm Hom}_{\mathcal{C}}(F, C) & 0 \\ };
\path[->]
(m-1-1) edge (m-1-2) (m-1-2) edge (m-1-3) (m-1-3) edge (m-1-4) edge (m-2-3) (m-1-4) edge (m-1-5) edge (m-2-4) (m-1-5) edge (m-1-6)
(m-2-2) edge (m-2-3) (m-2-3) edge (m-2-4) (m-2-4) edge (m-2-5) (m-2-5) edge (m-2-6);
\path[-,font=\scriptsize]
(m-1-5) edge [double, thick, double distance=4pt] (m-2-5);
\end{tikzpicture} \]
is commutative in the category of Abelian groups, where the top row is exact, so \[ {\rm Hom}_{\mathcal{C}}(F, D \coprod_{{\rm Ker}(f_0)} F_0) \longrightarrow {\rm Hom}_{\mathcal{C}}(F, C) \] is onto. Then 
\begin{align*}
h + {\rm Im}({\rm Hom}_{\mathcal{C}}(f_1, D)) & = h' \circ \widehat{f_1} + {\rm Im}({\rm Hom}_{\mathcal{C}}(f_1, D)) = \Phi( [0 \longrightarrow D \longrightarrow D \coprod_{{\rm Ker}(f_0)} F_0 \longrightarrow C \longrightarrow 0] ). 
\end{align*}
\end{proof}

\begin{remark} In the previous theorem, note that if $\mathcal{F}$ is a pre-covering class (not necessarily special), then the map $\Phi : \mathcal{E}xt^1_{\mathcal{C}}(\mathcal{F}; C, D) \longrightarrow {\rm Ext}^1_{\mathcal{C}}(\mathcal{F}; C, D)$ defined in the proof is a group monomorphism. The fact that $\mathcal{F}$ is special is used to show that $\Phi$ is also onto. 
\end{remark}

%%%%%%%%%%%%%%%%%%%%%%%%%%%%%%%%%%%%%%%%%%%%%%%%%%%%%%%%%%%%%%%%%%%%%%%%%%%%%%%%%%%%%%%%%%
%%%%%%%%%%%%%%%%%%%%%%%%%%%%%%%%%%%%%%%%%%%%%%%%%%%%%%%%%%%%%%%%%%%%%%%%%%%%%%%%%%%%%%%%%%

\section{Relative extensions and disk complexes}

Suppose we are given an Abelian category $\mathcal{C}$, an object $C$ in $\mathcal{C}$ and a chain complex $X$ over $\mathcal{C}$. In \cite[Lemma 3.1]{Gil}, J. Gillespie proved that the groups ${\rm Ext}^1_{\Complexes}(D^m(C),X)$ and ${\rm Ext}^1_{\mathcal{C}}(C,X_m)$ are naturally isomorphic, using the Baer description of these extension functors. On the other hand, it is also possible to prove this result describing Ext as right derived functors, as it appears in \cite[Proposition 2.1.3]{EJ2}. Dually, there exists a natural isomorphism between ${\rm Ext}^1_\Complexes(X,D^{m+1}(C))$ and ${\rm Ext}^1_{\mathcal{C}}(X_m, C)$. The goal of this section is to show that this isomorphisms have their versions in the context of relative extensions. We have to point out that if we are given a class $\mathcal{F}$ of objects of $\mathcal{C}$, then we need to consider an appropriate class of chain complexes induced $\mathcal{F}$. For our purposes, we are going to consider the following induced classes of chain complexes. \\

\begin{definition} Let $\mathcal{F}$ be a class of objects in an Abelian category $\mathcal{C}$. A chain complex $X$ over $\mathcal{C}$ is:
\begin{itemize}[noitemsep, topsep=-10pt]
\item[{\bf (1)}] {\rm \cite[Definition 3.3]{Gil}} An \underline{$\mathcal{F}$-complex} if $X$ is exact and $Z_m(X) \in \mathcal{F}$ for every $m \in \mathbb{Z}$.

\item[{\bf (2)}] {\rm \cite[Definition 3.1]{Gillespie}} A \underline{degreewise $\mathcal{F}$-complex} if $X_m \in \mathcal{F}$ for every $m \in \mathbb{Z}$. 
\end{itemize}
We shall denote by $\widetilde{\mathcal{F}}$ and ${\rm dw}\widetilde{\mathcal{F}}$ the classes of $\mathcal{F}$-complexes and degreewise $\mathcal{F}$-complexes, respectively. 
\end{definition}

Many interesting examples of (special) pre-covering and pre-enveloping classes of complexes are $\mathcal{F}$-complexes or degreewise $\mathcal{F}$-complexes, for some pre-covering or pre-enveloping class of objects $\mathcal{F}$. As a first example, recall that a chain complex is projective if it is a $\mathcal{P}roj(\mathcal{C})$-complex (see \cite[Theorem 10.42]{Rotman}). Moreover, $\mathcal{P}roj(\mathcal{C})$-complexes form a special pre-covering class if $\mathcal{C}$ has enough projective objects. The same applies to the class of $\mathcal{F}lat$-complexes, where $\mathcal{F}lat$ is the class of flat modules over a ring (see \cite[Corollary 4.10]{Gil}). Two other examples of pre-covering classes are given by the degreewise projective complexes (see \cite[Theorem 4.5]{Rada}) and the degreewise flat complexes (see \cite[Theorem 4.3]{Aldrich}). Dually, injective and degreewise injective complexes are pre-enveloping. We shall comment more examples in Section 6 in the setting provided by Gorenstein rings. Notice that from these examples it is natural to think that every pre-covering (pre-enveloping) class of objects induces a pre-covering (pre-enveloping) class of complexes. Complete cotorsion pairs provide a positive answer for special pre-covering (resp. special pre-enveloping) classes of modules (see \cite[Chapter 7]{EJ2}), but the author is not aware if this remains true in the non-special case. 

Using the Baer description presented in Section 3, we present the first generalization of \cite[Lemma 3.1]{Gil}. \\

\begin{proposition}\label{isosdw} Let $\mathcal{C}$ be an Abelian category and $\mathcal{F}$ and $\mathcal{G}$ be classes of objects of $\mathcal{C}$. If $C \in {\rm Ob}(\mathcal{C})$ and $X, Y \in \Complexes$, then we have natural isomorphisms: 
\begin{itemize}[noitemsep, topsep=-10pt]
\item[{\bf (1)}] $\mathcal{E}xt^i_{\mathcal{C}}(\mathcal{F}; X_m, C) \rightarrow  \mathcal{E}xt^i_{\Complexes}({\rm dw}\widetilde{\mathcal{F}}; X, D^{m+1}(C))$. 

\item[{\bf (2)}] $\mathcal{E}xt^i_{\mathcal{C}}(X_m, C; \mathcal{G}) \rightarrow \mathcal{E}xt^i_{\Complexes}(X,D^{m+1}(C); {\rm dw}\widetilde{\mathcal{G}})$.

\item[{\bf (3)}]$\mathcal{E}xt^i_{\mathcal{C}}(\mathcal{F}; C,Y_m) \rightarrow \mathcal{E}xt^i_{\Complexes}({\rm dw}\widetilde{\mathcal{F}}; D^m(C),Y)$.

\item[{\bf (4)}] $\mathcal{E}xt^i_{\mathcal{C}}(C, Y_m; \mathcal{G}) \rightarrow \mathcal{E}xt^i_{\Complexes}(D^m(C), Y; {\rm dw}\widetilde{\mathcal{G}})$.
\end{itemize}
\end{proposition}
\begin{proof}
We only give a proof for {\bf (1)} and {\bf (2)}, since {\bf (3)} and {\bf (4)} are dual. In order to present a shorter proof easy to understand, we only focus on the case $i = 1$, but the arguments given below also work for $i > 1$. We consider the maps constructed in \cite[Lemma 3.1]{Gil}.
\begin{itemize}[noitemsep, topsep=-10pt]
\item[{\bf (1)}] Let $[S] = [0 \longrightarrow D^{m+1}(C) \longrightarrow Z \longrightarrow X \longrightarrow 0] \in \mathcal{E}xt^1_{\Complexes}({\rm dw}\widetilde{\mathcal{F}}; X, D^{m+1}(C))$. Since the sequence $0 \longrightarrow D^{m+1}(C) \longrightarrow Z \longrightarrow X \longrightarrow 0$ is exact in $\Complexes$, we have $0 \longrightarrow C \longrightarrow Z_m \longrightarrow X_m \longrightarrow 0$ is exact in $\mathcal{C}$. We show it is also ${\rm Hom}_{\mathcal{C}}(\mathcal{F}, -)$-exact. For if $F \in \mathcal{F}$, then $D^m(F) \in {\rm dw}\widetilde{\mathcal{F}}$. We have the following commutative diagram: 
\[ \begin{tikzpicture}
\matrix (m) [matrix of math nodes, row sep=2em, column sep=2em, text height=1.5ex, text depth=0.25ex]
{ 0 & {\rm Hom}_{\mathcal{C}}(F, C) & {\rm Hom}_{\mathcal{C}}(F, Z_m) & {\rm Hom}_{\mathcal{C}}(F, X_m) & 0 \\ 0 & {\rm Hom}_{\Complexes}(D^m(F), D^{m+1}(C)) & {\rm Hom}_{\Complexes}(D^m(F), Z) & {\rm Hom}_{\Complexes}(D^m(F), X) & 0 \\ };
\path[->]
(m-1-1) edge (m-1-2) (m-1-2) edge (m-1-3) edge node[right] {$\cong$} (m-2-2) (m-1-3) edge (m-1-4) edge node[right] {$\cong$} (m-2-3) (m-1-4) edge (m-1-5) edge node[right] {$\cong$} (m-2-4)
(m-2-1) edge (m-2-2) (m-2-2) edge (m-2-3) (m-2-3) edge (m-2-4) (m-2-4) edge (m-2-5);
\end{tikzpicture} \]
Since the bottom row is exact and the vertical arrows are isomorphisms, we have that the top row is also exact. So $[0 \longrightarrow C \longrightarrow Z_m \longrightarrow X_m \longrightarrow 0] \in \mathcal{E}xt^1_{\mathcal{C}}(\mathcal{F}; X_m, C)$. So define a map $\Phi : \mathcal{E}xt^1_{\Complexes}({\rm dw}\widetilde{\mathcal{F}}; X, D^{m + 1}(C)) \longrightarrow \mathcal{E}xt^1_{\mathcal{C}}(\mathcal{F}; X_m, C)$ by setting $\Phi([S]) = \left[ 0 \longrightarrow C \longrightarrow Z_m \longrightarrow X_m \longrightarrow 0 \right]$. It is not hard to verify that $\Phi$ is a well defined group homomorphism.

Now we construct an inverse $\Psi : \mathcal{E}xt^1_{\mathcal{C}}(\mathcal{F}; X_m, C) \longrightarrow \mathcal{E}xt^1_{\Complexes}({\rm dw}\widetilde{F}; X, D^{m+1}(C))$ for $\Phi$. Consider a class $[S] = [0 \longrightarrow C \stackrel{\alpha}\longrightarrow Z \stackrel{\beta}\longrightarrow X_m \longrightarrow 0] \in \mathcal{E}xt^1_{\mathcal{C}}(\mathcal{F}; X_m, C)$. Taking the pullback of $\beta$ and $\partial^X_{m+1}$, we get the following commutative diagram with exact rows: 
\[ \begin{tikzpicture}
\matrix (m) [matrix of math nodes, row sep=2.5em, column sep=3em, text height=1.5ex, text depth=0.25ex]
{ 0 & C & Z \times_{X_m} X_{m+1} & X_{m+1} & 0 \\ 0 & C & Z & X_m & 0 \\ };
\path[->]
(m-1-1) edge (m-1-2) (m-1-2) edge node[above] {$\widetilde{\alpha}_{m+1}$} (m-1-3) edge node[right] {$=$} (m-2-2) (m-1-3) edge node[above] {$\widetilde{\beta}_{m+1}$} (m-1-4) edge node[right] {$\partial^{\widetilde{Z}}_{m+1}$} (m-2-3) (m-1-4) edge (m-1-5) edge node[right] {$\partial^X_{m + 1}$} (m-2-4)
(m-2-1) edge (m-2-2) (m-2-2) edge node[below] {$\alpha$} (m-2-3) (m-2-3) edge node[below] {$\beta$} (m-2-4) (m-2-4) edge (m-2-5);
\end{tikzpicture} \]
Let $\widetilde{Z}$ be the complex $\cdots \longrightarrow X_{m+2} \stackrel{\partial^X_{m+2}}\longrightarrow Z \times_{X_m} X_{m+1} \stackrel{\partial^{\widetilde{Z}}_{m+1}}\longrightarrow Z \longrightarrow X_{m-1} \longrightarrow \cdots$, where $\partial^{\widetilde{Z}}_m := \partial^X_{m} \circ \beta$, $\partial^{\widetilde{Z}}_{m + 2}$ is the map induced by the universal property of pullbacks satisfying $\widetilde{\beta}_{m+1} \circ \partial^{\widetilde{Z}}_{m+2} = \partial^{X}_{m+2}$, and $\partial^{\widetilde{Z}}_k = \partial^X_k$ for every $k \neq m, m+1, m+2$. From this we get an exact sequence of chain complexes $0 \longrightarrow D^{m+1}(C) \stackrel{\tilde{\alpha}}\longrightarrow \tilde{Z} \stackrel{\tilde{\beta}}\longrightarrow X \longrightarrow 0$, where $\tilde{\alpha}$ and $\tilde{\beta}$ are the chain maps given by: \[ \widetilde{\alpha}_k = \left\{ \begin{array}{ll} \alpha & \mbox{if $k = m$}, \\ \widetilde{\alpha}_{m+1} & \mbox{if $k = m+1$}, \\ 0 & \mbox{otherwise}. \end{array} \right. \mbox{ \ and \ } \widetilde{\beta}_k = \left\{ \begin{array}{ll} \beta & \mbox{if $k = m$}, \\ \widetilde{\beta}_{m+1} & \mbox{if $k = m + 1$}, \\ {\rm id}_{X_k} & \mbox{otherwise}. \end{array} \right. \] 
We prove that the previous sequence is ${\rm Hom}_{\Complexes}({\rm dw}\widetilde{\mathcal{F}}, -)$-exact. Let $F \in {\rm dw}\widetilde{\mathcal{F}}$ and suppose we are given a map $f : F \longrightarrow X$. We want to find a chain map $g : F \longrightarrow \tilde{Z}$ such that $\tilde{\beta} \circ g = f$. We set $g_k = f_k$ if $k \geq m + 2$ or $k \leq m - 1$. Since the sequence $0 \longrightarrow C \longrightarrow Z \longrightarrow X_m \longrightarrow 0$ is ${\rm Hom}_{\mathcal{C}}(\mathcal{F}, -)$-exact, there exists $g_m : F_m \rightarrow Z$ such that $\beta_m \circ g_m = f_m$. We have $\partial^{\widetilde{Z}}_m \circ g_m = \delta_m \circ g_m = \partial^X_m \circ \beta_m \circ g_m = \partial^X_m \circ f_m = f_{m - 1} \circ \partial^{F}_m = g_{m - 1} \circ \partial^{F}_m$. Now by the universal property of pullbacks, there exists a homomorphism $g_{m + 1} : F_{m + 1} \longrightarrow Z \times_{X_m} X_{m+1}$ such that the following diagram commutes:
\[ \begin{tikzpicture}
\matrix (m) [matrix of math nodes, row sep=3em, column sep=3em, text height=1.5ex, text depth=0.25ex]
{ F_{m+1} \\ & Z \times_{X_m} X_{m+1} & X_{m+1} \\ & Z & X_m \\ };
\path[->]
(m-1-1) edge [bend right=30] node[below,sloped] {$g_m \circ \partial^{F}_{m + 1}$} (m-3-2) edge [bend left=30] node[above,sloped] {$f_{m + 1}$} (m-2-3) edge node[above,sloped] {$g_{m + 1}$} (m-2-2)
(m-2-2) edge node[above,sloped] {$\widetilde{\beta}_{m + 1}$} (m-2-3) edge node[below,sloped] {$\partial^{\widetilde{Z}}_{m + 1}$} (m-3-2)
(m-3-2) edge node[below] {$\beta$} (m-3-3)
(m-2-3) edge node[above,sloped] {$\partial^X_{m + 1}$} (m-3-3);
\end{tikzpicture} \]
In order to show that $g = (g_k)_{k \in \mathbb{Z}}$ is a chain map, it is only left to show the equality $g_{m + 1} \circ \partial^{F}_{m + 1} = \partial^{\widetilde{Z}}_{m + 2} \circ g_{m + 2} = \partial^{\widetilde{Z}}_{m + 2} \circ f_{m + 2}$, which follows by the universal property of the previous pullback square. 
\[ \begin{tikzpicture}
\matrix (m) [matrix of math nodes, row sep=3.5em, column sep=3.5em]
{ F_{m+2} \\ & Z \times_{X_m} X_{m + 1} & X_{m + 1} \\ & Z & X_m \\ };
\path[->]
(m-1-1) edge [bend left=30] node[above,sloped] {$f_{m + 1} \circ \partial^{F}_{m + 2}$} (m-2-3) edge [bend right=40] node[below,sloped] {$0$} (m-3-2) edge [bend left=10] node[above,sloped] {$g_{m + 1} \circ \partial^{F}_{m + 1}$} (m-2-2) edge [bend right=10] node[below,sloped] {$\partial^{\widetilde{Z}}_{m + 2} \circ f_{m + 2}$} (m-2-2)
(m-2-2) edge node[above] {$\widetilde{\beta}_{m + 1}$} (m-2-3) edge node[below,sloped] {$\partial^{\widetilde{Z}}_{m + 1}$} (m-3-2)
(m-2-3) edge node[above,sloped] {$\partial^X_{m+1}$} (m-3-3)
(m-3-2) edge node[below] {$\beta$} (m-3-3);
\end{tikzpicture} \]
Then, we define a map $\Psi : \mathcal{E}xt^1_{\mathcal{C}}(\mathcal{F}; X_m, C) \longrightarrow \mathcal{E}xt^1_{\Complexes}({\rm dw}\widetilde{\mathcal{F}}; X, D^{m+1}(C))$ by setting $\Psi([S]) = [ 0 \longrightarrow D^{m+1}(C) \longrightarrow \widetilde{Z} \longrightarrow X \longrightarrow 0 ]$. It is not hard to see that $\Psi$ is a well defined group homomorphism such that ${\rm id}_{\mathcal{E}xt^1_{\Complexes}({\rm dw}\widetilde{F}; X, D^{m+1}(C))} = \Psi \circ \Phi$ and ${\rm id}_{\mathcal{E}xt^1_{\mathcal{C}}(\mathcal{F}; X_m, C)} = \Phi \circ \Psi$. \\

\item[{\bf (2)}] We use the same construction given in {\bf (1)}. Given a class $[0 \longrightarrow D^{m+1}(C) \longrightarrow Z \longrightarrow X \longrightarrow 0]$ in $\mathcal{E}xt^1_{\Complexes}(X, D^{m+1}(C); {\rm dw}\widetilde{\mathcal{G}})$, one can show as in {\bf (1)} that the sequence $0 \longrightarrow C \longrightarrow Z_m \longrightarrow X_m \longrightarrow 0$ is ${\rm Hom}_{\mathcal{C}}(-,\mathcal{G})$-exact.  

Now if we are given an exact and ${\rm Hom}_{\mathcal{C}}(-,\mathcal{G})$ sequence $0 \longrightarrow C \stackrel{\alpha}\longrightarrow Z \stackrel{\beta}\longrightarrow X_m \longrightarrow 0$, we show that the short exact sequence of complexes obtained by taking the pullback of $\beta$ and $\partial^X_{m + 1}$ is ${\rm Hom}_{\Complexes}(-, {\rm dw}\widetilde{\mathcal{G}})$-exact. Let $G \in {\rm dw}\widetilde{\mathcal{G}}$ and a chain map $f : D^{m+1}(C) \longrightarrow G$. We construct a chain map $h : \widetilde{Z} \longrightarrow G$ such that $h \circ \alpha = f$. For every $k \neq m, m + 1$, we set $h_k = 0$. Since the sequence $0 \longrightarrow C \longrightarrow Z \longrightarrow X_m \longrightarrow 0$ is ${\rm Hom}_{\mathcal{C}}(-,\mathcal{G})$-exact, there exists a map $h'_{m + 1} : Z \longrightarrow G_{m + 1}$ such that $f_{m + 1} = h'_{m + 1} \circ \alpha$. Set $h_{m + 1} := h'_{m + 1} \circ \partial^{\widetilde{Z}}_{m + 1}$ and $h_m := \partial^{G}_{m + 1} \circ h'_{m + 1}$. We have: 
\begin{align*}
h_{m+1} \circ \widetilde{\alpha}_{m+1} & = h'_{m+1} \circ \partial^{\widetilde{Z}}_{m+1} \circ \widehat{\alpha} = h'_{m+1} \circ \alpha = f_{m+1}, \\
h_m \circ \alpha & = \partial^{G}_{m+1} \circ h'_{m+1} \circ \alpha = \partial^{G}_{m+1} \circ f_{m+1} = f_m, \\
h_{m+1} \circ \partial^{\widetilde{Z}}_{m+2} & = h'_{m+1} \circ \partial^{\widetilde{Z}}_{m+1} \circ \partial^{\widetilde{Z}}_{m+2} = 0 = \partial^{G}_{m+1} \circ h_{m+2}, \\
h_m \circ \partial^{\widetilde{Z}}_{m+1} & = \partial^{G}_{m+1} \circ h'_{m+1} \circ \partial^{\widetilde{Z}}_{m+1} = \partial^{G}_{m+1} \circ h_{m+1}, \\
h_{m-1} \circ \partial^{\widetilde{Z}}_m & = 0 = \partial^{G}_m \circ \partial^{G}_{m+1} \circ h'_{m+1} = \partial^{G}_{m} \circ h_m.
\end{align*}
Hence, $h = (h_k \mbox{ : } k \in \mathbb{Z})$ is a chain map satisfying $h \circ \widetilde{\alpha} = f$. 
\end{itemize}
\end{proof}

Note that the complex $D^{m}(F)$ considered in the first part of the previous proof is actually a complex in $\widetilde{\mathcal{F}}$. So we can restrict $\Phi$ on $\mathcal{E}xt^1_{\Complexes}(\widetilde{\mathcal{F}}; X,D^{m+1}(C))$ to get a map $\mathcal{E}xt^1_{\Complexes}(\widetilde{\mathcal{F}}; X, D^{m+1}(C)) \longrightarrow \mathcal{E}xt^1_{\mathcal{C}}(\mathcal{F}; X_m, C)$, which is invertible in the case where $\mathcal{F}$ is \underline{closed under extensions}, i.e. that if for every short exact sequence $0 \longrightarrow F' \longrightarrow F \longrightarrow F'' \longrightarrow 0$ with $F', F'' \in \mathcal{F}$ one has $F \in \mathcal{F}$.

Under this hypothesis, we have that $F_n \in \mathcal{F}$ for every $F \in \widetilde{\mathcal{F}}$ (it suffices to consider the sequence $0 \longrightarrow Z_m(F) \longrightarrow F_m \longrightarrow Z_{m-1}(F) \longrightarrow 0$ for each $m \in \mathbb{Z}$). \\

\begin{proposition} Let $\mathcal{C}$ be an Abelian category and $\mathcal{F}$ and $\mathcal{G}$ be classes of objects of $\mathcal{C}$ which are closed under extensions. If $C \in {\rm Ob}(\mathcal{C})$ and $X, Y \in \Complexes$, then we have natural isomorphisms: 
\begin{itemize}[noitemsep, topsep=-10pt]
\item[{\bf (1)}] $\mathcal{E}xt^i_{\mathcal{C}}(\mathcal{F}; X_m, C) \cong \mathcal{E}xt^i_{\Complexes}(\widetilde{\mathcal{F}}; X, D^{m+1}(C))$. 

\item[{\bf (2)}] $\mathcal{E}xt^i_{\mathcal{C}}(X_m, C; \mathcal{G}) \cong \mathcal{E}xt^i_{\Complexes}(X, D^{m+1}(C); \widetilde{\mathcal{G}})$.

\item[{\bf (3)}] $\mathcal{E}xt^i_{\mathcal{C}}(\mathcal{F}; C, Y_m) \cong \mathcal{E}xt^i_{\Complexes}(\widetilde{\mathcal{F}}; D^m(C), Y)$.

\item[{\bf (4)}] $\mathcal{E}xt^i_{\mathcal{C}}(C, Y_m; \mathcal{G}) \cong \mathcal{E}xt^i_{\Complexes}(D^m(C), Y; \widetilde{\mathcal{G}})$. 
\end{itemize}
\end{proposition}

It follows that when $\mathcal{F}$ and $\mathcal{G}$ are closed under extensions, we have:
\begin{itemize}[noitemsep, topsep=-10pt] 
\item[{\bf (1)}] $\mathcal{E}xt^i_{\Complexes}(\widetilde{\mathcal{F}}; X, D^{m + 1}(C)) \cong \mathcal{E}xt^i_{\Complexes}({\rm dw}\widetilde{\mathcal{F}}; X, D^{m + 1}(C))$,

\item[{\bf (2)}] $\mathcal{E}xt^i_{\Complexes}(X, D^{m + 1}(C); \widetilde{\mathcal{G}}) \cong \mathcal{E}xt^i_{\Complexes}(X, D^{m + 1}(C); {\rm dw}\widetilde{\mathcal{G}})$, 

\item[{\bf (3)}] $\mathcal{E}xt^i_{\Complexes}(\widetilde{\mathcal{F}}; D^m(C), Y) \cong \mathcal{E}xt^i_{\Complexes}({\rm dw}\widetilde{\mathcal{F}}; D^m(C), Y)$, and 

\item[{\bf (4)}] $\mathcal{E}xt^i_{\Complexes}(D^m(C), Y; \widetilde{\mathcal{G}}) \cong \mathcal{E}xt^i_{\Complexes}(D^m(C), Y; {\rm dw}\widetilde{\mathcal{G}})$. \\
\end{itemize}

This seems to be a weird behaviour at a first glance, but this is clarified in the following proposition. \\

\begin{proposition}\label{relativedwsd} Let $\mathcal{C}$ be an Abelian category and $\mathcal{F}$ and $\mathcal{G}$ be classes of objects of $\mathcal{C}$ which are closed under extensions. Suppose we are given short exact sequences of the form $S = 0 \longrightarrow D^{m + 1}(C) \longrightarrow Z \longrightarrow X \longrightarrow 0$ and $S' = 0 \longrightarrow Y \longrightarrow Z \longrightarrow D^m(C) \longrightarrow 0$, for some integer $m \in \mathbb{Z}$. Then: 
\begin{itemize}[noitemsep, topsep=-10pt]
\item[{\bf (1)}] $S$ is ${\rm Hom}({\rm dw}\widetilde{\mathcal{F}},-)$-exact if, and only if, it is ${\rm Hom}(\widetilde{\mathcal{F}},-)$-exact.

\item[{\bf (2)}] $S$ is ${\rm Hom}(-,{\rm dw}\widetilde{\mathcal{G}})$-exact if, and only if, it is ${\rm Hom}(-,\widetilde{\mathcal{G}})$-exact.

\item[{\bf (3)}] $S'$ is ${\rm Hom}({\rm dw}\widetilde{\mathcal{F}},-)$-exact if, and only if, it is ${\rm Hom}(\widetilde{\mathcal{F}},-)$-exact.

\item[{\bf (4)}] $S'$ is ${\rm Hom}(-,{\rm dw}\widetilde{\mathcal{G}})$-exact if, and only if, it is ${\rm Hom}(-,\widetilde{\mathcal{G}})$-exact.
\end{itemize}
\end{proposition} 
\begin{proof} We only prove {\bf (1)}. The implication ($\Longrightarrow$) is clear, since $\widetilde{\mathcal{F}} \subseteq {\rm dw}\widetilde{\mathcal{F}}$ if $\mathcal{F}$ is closed under extensions. Now suppose $S$ is ${\rm Hom}_{\Complexes}( \widetilde{\mathcal{F}},- )$-exact. Note that $S$ and $0 \longrightarrow D^{m+1}(C) \longrightarrow \widetilde{Z_m} \longrightarrow X \longrightarrow 0$ are equivalent, so the result will follow if we show that the latter sequence is ${\rm Hom}_{\Complexes}({{\rm dw}\widetilde{\mathcal{F}}},-)$-exact. Since $0 \longrightarrow D^{m + 1}(C) \longrightarrow \widetilde{Z_m} \longrightarrow X \longrightarrow 0$ is ${\rm Hom}_{\Complexes}(\widetilde{\mathcal{F}},- )$-exact, we know $0 \longrightarrow C \longrightarrow Z_m \longrightarrow X_m \longrightarrow 0$ is ${\rm Hom}_{\mathcal{C}}(\mathcal{F},- )$-exact. Then as we did above, we can show that $0 \longrightarrow D^{m + 1}(C) \longrightarrow \widetilde{Z_m} \longrightarrow X \longrightarrow 0$ is ${\rm Hom}_{\Complexes}( {\rm dw}\widetilde{\mathcal{F}},- )$-exact.  
\end{proof}

%%%%%%%%%%%%%%%%%%%%%%%%%%%%%%%%%%%%%%%%%%%%%%%%%%%%%%%%%%%%%%%%%%%%%%%%%%%%%%%%%%%%%%%%%%
%%%%%%%%%%%%%%%%%%%%%%%%%%%%%%%%%%%%%%%%%%%%%%%%%%%%%%%%%%%%%%%%%%%%%%%%%%%%%%%%%%%%%%%%%%

\section{Relative extensions and sphere complexes}

In this section we study the connection between relative extensions and sphere chain complexes. In \cite[Lemma 4.2]{Gillespie}, J. Gillespie constructed two natural monomorphisms ${\rm Ext}^1_{\mathcal{C}}(C, Z_m(X)) \hookrightarrow {\rm Ext}^1_{\Complexes}(S^m(C), X)$ and ${\rm Ext}^1_{\mathcal{C}}(\frac{X_m}{B_m(X)}, C) \hookrightarrow {\rm Ext}^1_{\Complexes}(X,S^m(C))$, which are isomorphisms whether $X$ is exact. We shall prove similar results for relative extensions with respect to the classes $\widetilde{\mathcal{F}}$ and ${\rm dw}\widetilde{\mathcal{F}}$, for a given class $\mathcal{F}$ of objects of $\mathcal{C}$. \\

\begin{proposition} Let $\mathcal{C}$ be an Abelian category, and $\mathcal{F}$ and $\mathcal{G}$ be two classes of objects of $\mathcal{C}$ which are closed under extensions. Let $C \in {\rm Ob}(\mathcal{C})$ and $X, Y \in {\rm Ob}(\Complexes)$. There exist natural monomorphisms: 
\begin{itemize}[noitemsep, topsep=-10pt]
\item[{\bf (1)}] $\mathcal{E}xt^i_{\mathcal{C}}(\mathcal{F}; \frac{X_m}{B_m(X)}, C) \hookrightarrow \mathcal{E}xt^i_{\Complexes}(\widetilde{\mathcal{F}}; X, S^{m}(C))$. 

\item[{\bf (2)}] $\mathcal{E}xt^i_{\mathcal{C}}(\frac{X_m}{B_m(X)}, C; \mathcal{G}) \hookrightarrow \mathcal{E}xt^i_{\Complexes}(X, S^{m}(C); \widetilde{\mathcal{G}})$.

\item[{\bf (3)}] $\mathcal{E}xt^i_{\mathcal{C}}(\mathcal{F}; C, Z_m(Y)) \hookrightarrow \mathcal{E}xt^i_{\Complexes}(\widetilde{\mathcal{F}}; S^m(C), Y)$.

\item[{\bf (4)}] $\mathcal{E}xt^i_{\mathcal{C}}(C, Z_m(Y); \mathcal{G}) \hookrightarrow \mathcal{E}xt^i_{\Complexes}(S^m(C), Y; \widetilde{\mathcal{G}})$.
\end{itemize}
Moreover, if $X$ and $Y$ are exact and ${\rm Hom}_{\mathcal{C}}(\mathcal{F}, -)$-exact chain complexes, then {\bf (1)} and {\bf (3)} are invertible. Dually, the same is true for {\bf (2)} and {\bf (4)} if $X$ and $Y$ are exact and ${\rm Hom}_{\mathcal{C}}(-,\mathcal{G})$-exact. 
\end{proposition}
\begin{proof} We only prove the case $i = 1$ for the statements {\bf (1)} and {\bf (2)}.
\begin{itemize}[noitemsep, topsep=-10pt]
\item[{\bf (1)}] We consider the dual of the isomorphism given by J. Gillespie in \cite[Lemma 4.2]{Gillespie}. Suppopse we have an exact and ${\rm Hom}_{\mathcal{C}}(\mathcal{F},-)$-exact sequence $0 \longrightarrow C \stackrel{\alpha}\longrightarrow Z \stackrel{\beta}\longrightarrow \frac{X_m}{B_m(X)} \longrightarrow 0$. By taking the pullback of $\beta$ and $\pi^X_m : X_m \longrightarrow \frac{X_m}{B_m(X)}$, we construct a short exact sequence $0 \longrightarrow S^m(C) \stackrel{\widetilde{\alpha}}\longrightarrow \widetilde{Z} \stackrel{\widetilde{\beta}}\longrightarrow X \longrightarrow 0$ of chain complexes, where $\widetilde{\alpha}_k = 0$ and $\widetilde{\beta}_k = {\rm id}_{X_k}$ for every $k \neq m$, and $\widetilde{\alpha}_m$ and $\widetilde{\beta}_m$ are the morphisms appearing in the pullback diagram 
\[ \begin{tikzpicture}
\matrix (m) [matrix of math nodes, row sep=2em, column sep=3em, text height=1.5ex, text depth=0.25ex]
{ 0 & C & \widetilde{Z}_m & X_m & 0 \\ 0 & C & Z & \frac{X_m}{B_m(X)} & 0 \\ };
\path[->]
(m-1-1) edge (m-1-2) (m-1-2) edge node[above] {$\widetilde{\alpha}_m$} (m-1-3) (m-1-3) edge node[above] {$\widetilde{\beta}_m$} (m-1-4) edge node[right] {$\rho_Z$} (m-2-3) (m-1-4) edge (m-1-5) edge node[right] {$\pi^X_m$} (m-2-4)
(m-2-1) edge (m-2-2) (m-2-2) edge node[below] {$\alpha$} (m-2-3) (m-2-3) edge node[below] {$\beta$} (m-2-4) (m-2-4) edge (m-2-5);
\path[-,font=\scriptsize]
(m-1-2) edge [double, thick, double distance=4pt] (m-2-2);
\end{tikzpicture} \]
The arrow $\partial^{\widetilde{Z}}_{m + 1}$ is the map induced by the universal property of pullbacks such that $\widetilde{\beta}_m \circ \partial^{\widetilde{Z}}_{m + 1} = \partial^X_{m + 1}$ and $\rho_Z \circ \partial^{\widetilde{Z}}_{m + 1} = 0$, and $\partial^{\widetilde{Z}}_m := \partial^X_m \circ \widetilde{\beta}$. We show the sequence $0 \longrightarrow S^m(C) \stackrel{\widetilde{\alpha}}\longrightarrow \widetilde{Z} \stackrel{\widetilde{\beta}}\longrightarrow X \longrightarrow 0$ is also ${\rm Hom}_{\Complexes}( \widetilde{\mathcal{F}}, - )$-exact. Let $F \in \widetilde{\mathcal{F}}$ and consider a chain map $f : F \longrightarrow X$. We construct a chain map $h : F \longrightarrow \widetilde{Z}$ such that $\widetilde{\beta} \circ h = f$. Note that $\pi^X_m \circ f_m \circ \partial^{F}_{m + 1} = 0$. Factoring $\partial^{F}_{m + 1}$ as $i_{B_m(F)} \circ \widehat{\partial^{F}_{m + 1}}$, where $i_{B_m(F)} : B_m(F) \longrightarrow F_m$ is the inclusion and $\widehat{\partial^{F}_{m + 1}}$ is epic, we have that $\pi^X_m \circ f_m \circ i_{B_m(F)} = 0$. By the universal property of cokernels, there is a unique map $\overline{f_m} : \frac{F_m}{B_m(F)} \longrightarrow \frac{X_m}{B_m(X)}$ such that 
\[ \begin{tikzpicture}
\matrix (m) [matrix of math nodes, row sep=2em, column sep=3em, text height=1.5ex, text depth=0.25ex]
{ 0 & B_m(F) & F_m & \frac{F_m}{B_m(F)} & 0 \\ & & & \frac{X_m}{B_m(X)} \\ };
\path[->]
(m-1-3) edge node[below,sloped] {$\pi^X_m \circ f_m$} (m-2-4)
(m-1-2) edge node[above] {$i_{B_m(F)}$} (m-1-3)
(m-1-3) edge node[above] {$\pi^X_m$} (m-1-4)
(m-1-1) edge (m-1-2)
(m-1-4) edge (m-1-5);
\path[dotted,->]
(m-1-4) edge node[right] {$\exists \mbox{! } \overline{f_m}$} (m-2-4);
\end{tikzpicture} \] 
commutes. On the other hand, we have $\frac{F_m}{B_m(F)} \cong Z_{m - 1}(F) \in \mathcal{F}$. Since $0 \longrightarrow C \stackrel{\alpha}\longrightarrow Z \stackrel{\beta}\longrightarrow \frac{X_m}{B_m(X)} \longrightarrow 0$ is a ${\rm Hom}_{\mathcal{C}}(\mathcal{F}, -)$-exact sequence, there exists a morphism $h'_m : \frac{F_m}{B_m(F)} \longrightarrow Z$ such that \\
\[ \begin{tikzpicture}
\matrix (m) [matrix of math nodes, row sep=2em, column sep=3em, text height=1.5ex, text depth=0.25ex]
{ & & & \frac{F_m}{B_m(F)} \\ 0 & C & Z & \frac{X_m}{B_m(X)} & 0 \\ };
\path[->]
(m-2-1) edge (m-2-2) (m-2-4) edge (m-2-5)
(m-1-4) edge node[right] {$\overline{f_m}$} (m-2-4)
(m-2-2) edge node[below] {$\alpha$} (m-2-3)
(m-2-3) edge node[below] {$\beta$} (m-2-4);
\path[dotted,->]
(m-1-4) edge node[above,sloped] {$\exists \mbox{ } h'_m$} (m-2-3);
\end{tikzpicture} \]
commutes. Since $\mathcal{F}$ is closed under extensions and $F$ is exact, we have $F_m \in \mathcal{F}$. Using the universal property of pullbacks, we get the following commutative diagram 
\[ \begin{tikzpicture}
\matrix (m) [matrix of math nodes, row sep=2em, column sep=2em, text height=1.5ex, text depth=0.25ex]
{ F_m \\ & Z \times_{\frac{X_m}{B_m(X)}} X_m & X_m \\ & Z & \frac{X_m}{B_m(X)} \\ };
\path[->]
(m-1-1) edge [bend left=30] node[above,sloped] {$f_m$} (m-2-3) edge [bend right=30] node[below,sloped] {$h'_m \circ \pi^{F}_m$} (m-3-2)
(m-2-2) edge node[above,sloped] {$\widetilde{\beta}_m$} (m-2-3) edge node[left] {$\rho_Z$} (m-3-2)
(m-2-3) edge node[right] {$\pi^X_m$} (m-3-3)
(m-3-2) edge node[below] {$\beta$} (m-3-3);
\path[dotted,->]
(m-1-1) edge node[above,sloped] {$\exists \mbox{! } h_m$} (m-2-2);
\end{tikzpicture} \]
Set $h_k = f_k$ for every $k \neq m$. We have $\widetilde{\beta}_k \circ h_k = f_k$ for every $k \in \mathbb{Z}$. We check $h = (h_k \mbox{ : } k \in \mathbb{Z})$ is a chain map. The equality $h_m \circ \partial^{F}_{m + 1} = \partial^{\widetilde{Z}}_{m + 1} \circ f_{m + 1}$ follows by the commutativity of the following diagram: 
\[ \begin{tikzpicture}
\matrix (m) [matrix of math nodes, row sep=2.5em, column sep=4em, text height=1.5ex, text depth=0.25ex]
{ F_{m + 1} \\ & Z \times_{\frac{X_m}{B_m(X)}} X_m & X_m \\ & Z & \frac{X_m}{B_m(X)} \\ };
\path[->]
(m-1-1) edge [bend left=10] node[above,sloped] {$h_m \circ \partial^{F}_{m + 1}$} (m-2-2) edge [bend right=10] node[below,sloped] {$\partial^{\widetilde{Z}}_{m + 1} \circ f_{m + 1}$} (m-2-2) edge [bend left=30] node[above,sloped] {$f_m \circ \partial^{F}_{m + 1}$} (m-2-3) edge [bend right=50] node[below,sloped] {$0$} (m-3-2)
(m-2-2) edge node[above] {$\widetilde{\beta}_m$} (m-2-3) edge node[left] {$\rho_Z$} (m-3-2)
(m-3-2) edge node[below] {$\beta$} (m-3-3)
(m-2-3) edge node[right] {$\pi^X_m$} (m-3-3);
\end{tikzpicture} \]
On the other hand, $\partial^{\widetilde{Z}}_m \circ h_m = \partial^X_m \circ \widetilde{\beta}_m \circ h_m = \partial^X_m \circ f_m = f_{m - 1} \circ \partial^{F}_m$. Therefore, $h$ is a chain map satisfying $\widetilde{\beta} \circ h = f$.

Since the map $\mathcal{E}xt^1_{\mathcal{C}}(\frac{X_m}{B_m(X)}, C) \longrightarrow \mathcal{E}xt^1_{\Complexes}(X, S^m(C))$ constructed by Gillespie is monic, so is the restriction $\mathcal{E}xt^1_{\mathcal{C}}(\mathcal{F}; \frac{X_m}{B_m(X)}, C) \rightarrow \mathcal{E}xt^1_{\Complexes}(\widetilde{\mathcal{F}}; X, S^{m}(C))$. \\ 

\item[{\bf (2)}] Suppose $0 \longrightarrow C \longrightarrow \frac{Z_m}{B_m(Z)} \longrightarrow \frac{X_m}{B_m(X)} \longrightarrow 0$ is an exact and ${\rm Hom}_{\mathcal{C}}(-,\mathcal{G})$-exact sequence. Let $G \in \widetilde{\mathcal{G}}$ and consider a chain map $f : S^m(C) \longrightarrow G$. We construct a chain map $h : \widetilde{Z} \longrightarrow G$ such that $h \circ \widetilde{\alpha} = f$.

Since $\partial^{G}_m \circ f_m = 0$, there exists a unique map $\overline{f_m}$ in completing the following commutative diagram: \\
\[ \begin{tikzpicture}
\matrix (m) [matrix of math nodes, row sep=2em, column sep=2.5em, text height=1.5ex, text depth=0.25ex]
{ 0 & Z_m(G) & G_m & Z_{m - 1}(G) & 0 \\ & C \\ };
\path[->]
(m-1-1) edge (m-1-2)
(m-1-4) edge (m-1-5)
(m-2-2) edge node[below,sloped] {$f_m$} (m-1-3)
(m-1-2) edge node[above] {$i_{Z_m(G)}$} (m-1-3)
(m-1-3) edge node[above] {$\widehat{\partial^{G}_m}$} (m-1-4);
\path[dotted,->]
(m-2-2) edge node[left] {$\overline{f_m}$} (m-1-2);
\end{tikzpicture} \]
On the other hand, $0 \longrightarrow C \longrightarrow Z \longrightarrow \frac{X_m}{B_m(X)} \longrightarrow 0$ is ${\rm Hom}_{\mathcal{C}}(-,\mathcal{G})$-exact and $Z_m(G) \in \mathcal{G}$, so there is a morphism $Z \stackrel{h'_m}\longrightarrow Z_m(G)$ such that $h'_m \circ \alpha = \overline{f_m}$. Set $h_k = 0$ for every $k \neq m$ and $h_m := i_{Z_m(G)} \circ h'_m \circ \rho_{Z}$.
\begin{align*}
\partial^{G}_m \circ h_m & = 0 = h_{m - 1} \circ \partial^{\widetilde{Z}}_m, \\
h_m \circ \partial^{\widetilde{Z}}_{m + 1} & = i_{Z_m(G)} \circ h'_m \circ \rho_{Z} \circ \partial^{\widetilde{Z}}_{m + 1} = 0 = \partial^{G}_{m + 1} \circ h_{m + 1}, \\
h_m \circ \widetilde{\alpha}_m & = i_{Z_m(G)} \circ h'_m \circ \rho_{Z} \circ \widetilde{\alpha}_m = i_{Z_m(G)} \circ h'_m \circ \alpha = i_{Z_m(G)} \circ \overline{f_m} = f_m.
\end{align*}
Hence, $h$ is a chain map satisfying $h \circ \widetilde{\alpha} = f$. The map $\mathcal{E}xt^1_{\mathcal{C}}(\frac{X_m}{B_m(X)}, C; \mathcal{G}) \hookrightarrow \mathcal{E}xt^1_{\Complexes}(X, S^m(C); \widetilde{\mathcal{G}})$ is  a group monomorphism. 
\end{itemize}
For the last part of the statement, we prove that {\bf (1)} and {\bf (3)} are invertible if $X$ and $Y$ are exact and ${\rm Hom}_{\mathcal{C}}(\mathcal{F},-)$-exact. The arguments we present below are dual for {\bf (2)} and {\bf (4)}. 
\begin{itemize}
\item[{\bf (1)}] Suppose we are given a short exact sequence $0 \longrightarrow S^m(C) \longrightarrow Z \longrightarrow X \longrightarrow 0$. By \cite[Lemma 3.2]{Perez}, the induced sequence $0 \longrightarrow C \longrightarrow \frac{Z_m}{B_m(Z)} \longrightarrow \frac{X_m}{B_m(X)} \longrightarrow 0$ is exact since $X$ is exact. This defines the inverse of the map $\mathcal{E}xt^1_{\mathcal{C}}(\mathcal{F}; \frac{X_m}{B_m(X)}, C) \hookrightarrow \mathcal{E}xt^1_{\Complexes}(\widetilde{\mathcal{F}}; X, S^{m}(C))$. It is only left to show that $0 \longrightarrow C \longrightarrow \frac{Z_m}{B_m(Z)} \longrightarrow \frac{X_m}{B_m(X)} \longrightarrow 0$ is also ${\rm Hom}_{\mathcal{C}}(\mathcal{F},-)$-exact. 

First, note the exact sequence of $m$th cycles $0 \longrightarrow Z_m(X) \longrightarrow X_m \longrightarrow Z_{m-1}(X) \longrightarrow 0$ is ${\rm Hom}_{\mathcal{C}}(\mathcal{F},-)$-exact. For $\cdots \longrightarrow {\rm Hom}_{\mathcal{C}}(F,X_{m+1}) \longrightarrow {\rm Hom}_{\mathcal{C}}(F,X_m) \longrightarrow {\rm Hom}_{\mathcal{C}}(F,X_{m-1}) \longrightarrow \cdots$ is exact for every $F \in \mathcal{F}$, and so $0 \longrightarrow {\rm Ker}({\rm Hom}_{\mathcal{C}}(F,\partial^X_{m})) \longrightarrow {\rm Hom}_{\mathcal{C}}(F,X_m) \longrightarrow {\rm Ker}({\rm Hom}_{\mathcal{C}}(F,\partial^X_{m-1})) \longrightarrow 0$ is exact for every $m \in \mathbb{Z}$. On the other hand, it is not hard to see that ${\rm Ker}({\rm Hom}_{\mathcal{C}}(F,\partial^X_{m})) \cong {\rm Hom}_{\mathcal{C}}(F,Z_m(X))$.  

Since the sequence $0 \longrightarrow S^m(C) \longrightarrow Z \stackrel{g}\longrightarrow X \longrightarrow 0$ is exact and ${\rm Hom}(\widetilde{\mathcal{F}},-)$-exact, we can deduce $0 \longrightarrow C \stackrel{f_m}\longrightarrow Z_m \stackrel{g_m}\longrightarrow X_m \longrightarrow 0$ is ${\rm Hom}_{\mathcal{C}}(\mathcal{F},-)$-exact, by considering disk complexes $D^m(F)$ with $F \in \mathcal{F}$ as in the proof of Proposition \ref{isosdw}. We show $0 \longrightarrow C \longrightarrow \frac{Z_m}{B_m(Z)} \stackrel{\overline{g_m}}\longrightarrow \frac{X_m}{B_m(X)} \longrightarrow 0$ is also ${\rm Hom}_{\mathcal{C}}(\mathcal{F},-)$-exact. Consider a map $h : F \longrightarrow \frac{X_m}{B_m(X)}$ with $F \in \mathcal{F}$. By the previous comments, there exists a map $h' : F \longrightarrow X_m$ such that \\
\[ \begin{tikzpicture}
\matrix (m) [matrix of math nodes, row sep=2em, column sep=2em, text height=1.5ex, text depth=0.25ex]
{ & & & F \\ 0 & B_m(X) & X_m & \frac{X_m}{B_m(X)} & 0 \\ };
\path[->]
(m-1-4) edge node[right] {$h$} (m-2-4)
(m-2-1) edge (m-2-2) (m-2-2) edge (m-2-3) (m-2-3) edge node[below] {$\pi^X_m$} (m-2-4) (m-2-4) edge (m-2-5);
\path[dotted,->]
(m-1-4) edge node[above,sloped] {$\exists\mbox{ } h'$} (m-2-3);
\end{tikzpicture} \]
commutes. It follows the existence of a map $h'' : F \longrightarrow Z_m$ making the following diagram commute: \\
\[ \begin{tikzpicture}
\matrix (m) [matrix of math nodes, row sep=2em, column sep=2em, text height=1.5ex, text depth=0.25ex]
{ & & & F \\ 0 & C & Z_m & X_m & 0 \\ };
\path[->]
(m-1-4) edge node[right] {$h'$} (m-2-4)
(m-2-1) edge (m-2-2) (m-2-2) edge (m-2-3) (m-2-3) edge node[below] {$g_m$} (m-2-4) (m-2-4) edge (m-2-5);
\path[dotted,->]
(m-1-4) edge node[above,sloped] {$\exists\mbox{ }h''$} (m-2-3);
\end{tikzpicture} \]
We $\overline{g_m} \circ (\pi^Z_m \circ h'') = \pi^X_m \circ g_m \circ h'' = \pi^X_m \circ h' = h$, and hence $0 \longrightarrow C \longrightarrow \frac{Z_m}{B_m(Z)} \stackrel{\overline{g_m}}\longrightarrow \frac{X_m}{B_m(X)} \longrightarrow 0$ is ${\rm Hom}_{\mathcal{C}}(\mathcal{F},-)$-exact.

\item[{\bf (3)}] First, the map $\mathcal{E}xt^1_{\mathcal{C}}(C,Z_m(Y)) \hookrightarrow \mathcal{E}xt^1_{\Complexes}(S^m(C), Y)$ is invertible if $Y$ is exact, since every short exact sequence $0 \longrightarrow Y \longrightarrow Z \longrightarrow S^m(C) \longrightarrow 0$ induces in $\mathcal{C}$ a short exact sequence of cycles $0 \longrightarrow Z_m(Y) \longrightarrow Z_m(Z) \longrightarrow C \longrightarrow 0$, by \cite[Lemma 3.2]{Perez}. Consider $F \in \mathcal{F}$ and suppose that the sequence $0 \longrightarrow Y \longrightarrow Z \longrightarrow S^m(C) \longrightarrow 0$ is ${\rm Hom}_{\Complexes}(\widetilde{\mathcal{F}}, -)$-exact. We show that the sequence $0 \longrightarrow {\rm Hom}_{\mathcal{C}}(F, Z_m(Y)) \longrightarrow {\rm Hom}_{\mathcal{C}}(F, Z_m(Z)) \longrightarrow {\rm Hom}_{\mathcal{C}}(F, C) \longrightarrow 0$ is exact. Consider the following commutative grid
\[ \begin{tikzpicture}
\matrix (m) [matrix of math nodes, row sep=2em, column sep=2em, text height=1.5ex, text depth=0.25ex]
{ & 0 & 0 & 0 \\ 0 & {\rm Hom}_{\mathcal{C}}(F, Z_m(X)) & {\rm Hom}_{\mathcal{C}}(F, Z_m(Z)) & {\rm Hom}_{\mathcal{C}}(F, C) & 0 \\ 0 & {\rm Hom}_{\mathcal{C}}(F, X_m) & {\rm Hom}_{\mathcal{C}}(F, Z_m) & {\rm Hom}_{\mathcal{C}}(F, C) & 0 \\ 0 & {\rm Hom}_{\mathcal{C}}(F, Z_{m-1}(X)) & {\rm Hom}_{\mathcal{C}}(F, Z_{m-1}(Z)) & 0 \\ & 0 & 0 \\ };
\path[->]
(m-1-2) edge (m-2-2) (m-1-3) edge (m-2-3) (m-1-4) edge (m-2-4)
(m-2-1) edge (m-2-2) (m-2-2) edge (m-2-3) edge (m-3-2) (m-2-3) edge (m-2-4) edge (m-3-3) (m-2-4) edge (m-2-5)
(m-3-1) edge (m-3-2) (m-3-2) edge (m-3-3) edge (m-4-2) (m-3-3) edge (m-3-4) edge (m-4-3) (m-3-4) edge (m-3-5) edge (m-4-4)
(m-4-1) edge (m-4-2) (m-4-2) edge node[below] {$\cong$} (m-4-3) edge (m-5-2) (m-4-3) edge (m-4-4) edge (m-5-3);
\path[-,font=\scriptsize]
(m-2-4) edge [double, thick, double distance=4pt] (m-3-4);
\end{tikzpicture} \]
where the central row is exact for being naturally isomorphic to sequence \[ 0 \longrightarrow {\rm Hom}_{\Complexes}(D^m(F),X) \longrightarrow {\rm Hom}_{\Complexes}(D^m(F), Z) \longrightarrow {\rm Hom}_{\Complexes}(D^m(F),S^m(C)) \longrightarrow 0, \] which is exact since $D^m(F) \in \widetilde{\mathcal{F}}$. The bottom row and the rightmost column are clearly exact. On the other hand, the leftmost column is exact since $X$ is an exact and ${\rm Hom}_{\mathcal{C}}(\mathcal{F},-)$-exact complex. Consider an arrow $r : F \longrightarrow Z_{m-1}(Z)$. Notice that $f_{m-1} : X_{m-1} \longrightarrow Z_{m-1}$ is an isomorphism, and so is $Z_{m-1}(f) : Z_{m-1}(X) \longrightarrow Z_{m-1}(Z)$. Then we have an arrow $Z_{m-1}(f)^{-1} \circ r : F \longrightarrow Z_{m-1}(X)$. Since the leftmost column is exact, there exists an arrow $l : F \longrightarrow X_m$ such that $Z_{m-1}(f)^{-1} \circ r = \pi^X_m \circ l$, where $\rho^X_m$ is the arrow $X_m \longrightarrow Z_{m-1}(X)$ induced by the universal property of kernels. Consider $f_m \circ l : F \longrightarrow Z_m$. We have $\rho^Z_m \circ (f_m \circ l) = Z_{m-1}(f) \circ \rho^X_m \circ l = Z_{m-1}(f) \circ Z_{m-1}(f)^{-1} \circ r = r$, and hence ${\rm Hom}_{\mathcal{C}}(F, Z_m) \longrightarrow {\rm Hom}_{\mathcal{C}}(F, Z_{m-1}(Z))$ is surjective. Finally, using diagram chasing, one can show that the top row is also exact. 
\end{itemize}
\end{proof}

\begin{proposition}\label{isospheres} Let $\mathcal{C}$ be an Abelian category. Let $C \in {\rm Ob}(\mathcal{C})$ and $X$ and $Y$ be exact chain complexes. There exist natural monomorphisms: 
\begin{itemize}[noitemsep, topsep=-10pt]
\item[{\bf (1)}] $\mathcal{E}xt^i_{\Complexes}({\rm dw}\widetilde{\mathcal{F}}; S^m(C), Y) \hookrightarrow \mathcal{E}xt^i_{\mathcal{C}}(\mathcal{F}; C, Z_m(Y))$.

\item[{\bf (2)}] $\mathcal{E}xt^i_{\Complexes}(X, S^m(C); {\rm dw}\widetilde{\mathcal{G}}) \hookrightarrow \mathcal{E}xt^i_{\mathcal{C}}(\frac{X_m}{B_m(X)}, C; \mathcal{G})$. 

\item[{\bf (3)}] $\mathcal{E}xt^i_{\Complexes}({\rm dw}\widetilde{\mathcal{F}}; X, S^m(C)) \hookrightarrow \mathcal{E}xt^i_{\mathcal{C}}(\mathcal{F}; \frac{X_m}{B_m(X)}, C)$ provided $X$ is ${\rm Hom}_{\mathcal{C}}(\mathcal{F},-)$-exact.

\item[{\bf (4)}] $\mathcal{E}xt^i_{\Complexes}(S^m(C), Y; {\rm dw}\widetilde{\mathcal{G}}) \hookrightarrow \mathcal{E}xt^i_{\mathcal{C}}(C, Z_m(Y); \mathcal{G})$ provided $Y$ is ${\rm Hom}_{\mathcal{C}}(-,\mathcal{G})$-exact.
\end{itemize}
\end{proposition}
\begin{proof}
We only prove {\bf (2)} and {\bf (3)} for the case $i = 1$. We know by \cite[Lemma 4.2]{Gillespie} that the mapping \[ 0 \longrightarrow S^m(C) \longrightarrow Z \longrightarrow X \longrightarrow 0 \mbox{ \ $\mapsto$ \ } 0 \longrightarrow C \longrightarrow \frac{Z_m}{B_m(Z)} \longrightarrow \frac{X_m}{B_m(X)} \longrightarrow 0 \] gives rise to an isomorphism $\mathcal{E}xt^1_{\Complexes}(X, S^m(C)) \hookrightarrow \mathcal{E}xt^1_{\mathcal{C}}(\frac{X_m}{B_m(X)}, C)$, since $X$ is exact. It suffices to show that its restriction on $\mathcal{E}^1_{\Complexes}(X, S^m(C); {\rm dw}\widetilde{\mathcal{G}})$ is well defined. So consider an exact and ${\rm Hom}_{\Complexes}(-,{\rm dw}\widetilde{\mathcal{G}})$-exact sequence $0 \longrightarrow S^n(M) \longrightarrow Z \longrightarrow X \longrightarrow 0$. If $G \in \mathcal{G}$, then $S^n(G) \in {\rm dw}\widetilde{G}$. We have the following commutative diagram
\[ \begin{tikzpicture}
\matrix (m) [matrix of math nodes, row sep=2.5em, column sep=2em, text height=1.5ex, text depth=0.25ex]
{ 0 & {\rm Hom}_{\Complexes}(X,S^n(G)) & {\rm Hom}_{\Complexes}(Z,S^n(G)) & {\rm Hom}_{\Complexes}(S^n(M),S^n(G)) & 0 \\ 0 & {\rm Hom}_{\mathcal{C}}( \frac{X_n}{B_n(X)}, G ) & {\rm Hom}_{\mathcal{C}}( \frac{Z_n}{B_n(Z)}, G ) & {\rm Hom}_{\mathcal{C}}(M, G) & 0 \\ };
\path[->]
(m-1-1) edge (m-1-2) (m-1-2) edge (m-1-3) edge node[right] {$\cong$} (m-2-2) (m-1-3) edge (m-1-4) edge node[right] {$\cong$} (m-2-3) (m-1-4) edge (m-1-5) edge node[right] {$\cong$} (m-2-4)
(m-2-1) edge (m-2-2) (m-2-2) edge (m-2-3) (m-2-3) edge (m-2-4) (m-2-4) edge (m-2-5);
\end{tikzpicture} \]
where the top row is exact. It follows the bottom row is also exact. 

Now suppose that the sequence $0 \longrightarrow S^m(C) \longrightarrow Z \stackrel{f}\longrightarrow X \longrightarrow 0$ is ${\rm Hom}_{\Complexes}({\rm dw}\widetilde{\mathcal{F}},-)$-exact and that $X$ is ${\rm Hom}_{\mathcal{C}}(\mathcal{F},-)$-exact. Given $F \in \mathcal{F}$ and an arrow $h : F \longrightarrow \frac{X_m}{B_m(X)}$, we construct an arrow $F \longrightarrow \frac{Z_m}{B_m(Z)}$ such that the following diagram commutes:
\[ \begin{tikzpicture}
\matrix (m) [matrix of math nodes, row sep=2em, column sep=3em, text height=1.5ex, text depth=0.25ex]
{ & & & F \\ 0 & C & \frac{Z_m}{B_m(Z)} & \frac{X_m}{B_m(X)} & 0 \\ };
\path[->]
(m-1-4) edge node[right] {$h$} (m-2-4)
(m-2-1) edge (m-2-2) (m-2-2) edge (m-2-3) (m-2-3) edge node[below] {$Q_m(f)$} (m-2-4) (m-2-4) edge (m-2-5);
\path[dotted,->]
(m-1-4) edge (m-2-3);
\end{tikzpicture} \]
Since $X$ is ${\rm Hom}_{\mathcal{C}}(\mathcal{F},-)$-exact, so is the sequence $0 \longrightarrow B_m(X) \longrightarrow X_m \stackrel{\pi^X_m}\longrightarrow \frac{X_m}{B_m(X)} \longrightarrow 0$, and hence there exists an arrow $h' : F \longrightarrow X_m$ such that $\pi^X_m \circ h' = h$. Considering the complex $D^m(F)$, we can deduce that $0 \longrightarrow C \longrightarrow Z_m \longrightarrow X_m \longrightarrow 0$ is ${\rm Hom}_{\mathcal{C}}(\mathcal{F}, -)$-exact, as in the proof of Proposition \ref{isosdw}. It follows there exists an arrow $h'' : F \longrightarrow Z_m$ such that $f \circ h'' = h'$. Finally, we have that $Q_m(f) \circ (\pi^Z_m \circ h'') = \pi^X_m \circ f_m \circ h'' = \pi^X_m \circ h' = h$ and the result follows.   
\end{proof}

\section{Applications to Gorenstein homological algebra}

In this section we shall work in the particular case where $\mathcal{C}$ is the category of left $R$-modules, with $R$ a \underline{Gorenstein ring}; i.e. $R$ is left and right Noetherian and has finite injective dimension as a left and right $R$-module. It can be shown that both dimensions coincide to a non-negative integer $n$, and in this case we say $R$ is an $n$-Gorenstein ring. In this particular setting, another theory of homological algebra can be developed from the notions of Gorenstein-projective and Gorenstein-injective modules (and complexes). 

If $R$ is a $n$-Gorenstein ring, then the the following conditions are equivalent for every left $R$-module $M$: 
\begin{itemize}[noitemsep, topsep=-10pt]
\item[{\bf (1)}] $M$ has finite projective dimension.
\item[{\bf (2)}] $M$ has finite injective dimension.
\item[{\bf (3)}] $M$ has projective dimension $\leq n$.
\item[{\bf (4)}] $M$ has injective dimension $\leq n$. 
\end{itemize}
This fact was proven by Y. Iwanaga, and it is after him that Gorenstein rings are also known as $n$-Iwanaga-Gorenstein rings. The reader can see the details in \cite[Theorem 9.1.10]{EJ}. We shall denote by $\mathcal{W}$ the class of modules with finite projective dimension. For our purposes, a module over a Gorenstein ring is \underline{Gorenstein-projective} if ${\rm Ext}^1_R(M,W) = 0$ for every $W \in \mathcal{W}$. Gorenstein-injective modules are defined dually. 

If $R$ is a Gorenstein ring, it is known that the the class $\mathcal{GP}roj$ of Gorenstein-projective modules is special pre-covering (See \cite[Theorem 11.5.1]{EJ}). On the other hand, the class $\mathcal{GI}nj$ of Gorenstein-injective modules is special pre-enveloping (See \cite[Theorem 11.3.2]{EJ}). 

Consider the extension functors ${\rm Ext}^i_{\Modl}(\mathcal{GP}roj; -, -)$ and ${\rm Ext}^i_{\Modl}(-,-; \mathcal{GI}nj)$. In \cite[Theorem 12.1.4]{EJ}, it is proven that these functors are naturally isomorphic. So we shall use the notation ${\rm GExt}^i_R(-,-)$ for both ${\rm Ext}^i_{\Modl}(\mathcal{GP}roj; -, -)$ and ${\rm Ext}^i_{\Modl}(-,-; \mathcal{GI}nj)$. We shall call ${\rm GExt}^i_R(-,-)$ the \underline{Gorenstein-extension} \underline{functors}. By Theorem \ref{isosext}, we obtain the following result. \\

\begin{corollary}\label{isosGext} If $R$ is a Gorenstein ring, then for every pair of left $R$-modules $M$ and $N$ one has the isomorphisms ${\rm GExt}^i_R(M, N) \cong \mathcal{E}xt^i_{\Modl}(\mathcal{GP}roj; M, N) \cong \mathcal{E}xt^i_{\Modl}(M, N; \mathcal{GI}nj)$. 
\end{corollary}

In the context of chain complexes over a Gorenstein ring, $X$ is a Gorenstein-projective (resp. Gorenstein-injective) complex if it is left orthogonal (resp. right orthogonal) to every complex with finite projective dimension. By \cite[Proposition 3.1]{Perez}, we can notice that this class is given by $\widetilde{\mathcal{W}}$. The classes of Gorenstein-projective and Gorenstein-injective complexes are special pre-covering and special pre-enveloping, respectively (See \cite[Theorem 3.2.9 \& Corollary 3.3.7]{GR}). Moreover, these classes coincide with ${\rm dw}\widetilde{\mathcal{GP}roj}$ and ${\rm dw}\widetilde{\mathcal{GI}nj}$ (See \cite[Theorem 3.2.5 \&  Theorem 3.3.5]{GR}). From these comments and Theorem \ref{isosext}, the following result follows. However, in the context of Gorenstein homological algebra it is possible to present an easier proof. \\
 
\begin{corollary} Let $M$ be a left module over a Gorenstein ring $R$, and $X$ and $Y$ be two chain complexes over $R$. We have the following natural isomorphisms: 
\begin{itemize}[noitemsep, topsep=-10pt]
\item[{\bf (1)}] ${\rm GExt}^i_R(X_m, M) \cong {\rm GExt}^i_{\Cadl}(X, D^{m + 1}(M))$. 

\item[{\bf (2)}] ${\rm GExt}^i_R(M, Y_m) \cong {\rm GExt}^i_{\Cadl}(D^m(M), Y)$.
\end{itemize}
\end{corollary}
\begin{proof} We only reprove {\bf (2)}. The argument we present next is based on \cite[Proposition 2.1.3]{EJ2}. Consider an exact left Gorenstein-projective resolution of $M$, say $\cdots \longrightarrow C_1 \longrightarrow C_0 \longrightarrow M \longrightarrow 0$. Since ${\rm dw}\widetilde{\mathcal{GP}roj}$ is the class of Gorenstein-projective complexes and $D^{m}(-)$ is an exact functor, we have that the complex $\cdots \longrightarrow D^m(C_1) \longrightarrow D^m(C_0) \longrightarrow D^m(M) \longrightarrow 0$ is an exact left Gorenstein-projective resolution of $D^m(M)$. We obtain the following commutative diagram where each vertical arrow is an isomorphism. 
\[ \begin{tikzpicture}
\matrix (m) [matrix of math nodes, row sep=2em, column sep=2em, text height=1.5ex, text depth=0.25ex]
{ 0 & {\rm Hom}_{\Cadl}(D^m(M), Y) & {\rm Hom}_{\Cadl}(D^m(C_0), Y) & {\rm Hom}_{\Cadl}(D^m(C_1, Y)) & \cdots \\ 0 & {\rm Hom}_{R}(M, Y_m) & {\rm Hom}_R(C_0, Y_m) & {\rm Hom}_R(C_1, Y_m) & \cdots \\ };
\path[->]
(m-1-1) edge (m-1-2) (m-1-2) edge (m-1-3) edge node[right] {$\cong$} (m-2-2) (m-1-3) edge (m-1-4) edge node[right] {$\cong$} (m-2-3) (m-1-4) edge (m-1-5) edge node[right] {$\cong$} (m-2-4)
(m-2-1) edge (m-2-2) (m-2-2) edge (m-2-3) (m-2-3) edge (m-2-4) (m-2-4) edge (m-2-5);
\end{tikzpicture} \] 
Isomorphic complexes have isomorphic homology, then ${\rm GExt}^i_R(M, Y_m) \cong {\rm GExt}^i_{\Cadl}(D^m(M), Y)$. 
\end{proof}

For a Gorenstein version of Proposition \ref{isospheres}, we do not need to assume that $X$ is ${\rm Hom}_{\mathcal{C}}(\mathcal{GP}roj,-)$-exact. We shall see the reason in the proof of the following proposition. \\

\begin{proposition}\label{isosgorspheres} Let $M$ be a left module over a Gorenstein ring $R$, and $X$ and $Y$ be exact chain complexes over $R$. Then we have natural isomorphisms:
\begin{itemize}[noitemsep, topsep=-10pt]
\item[{\bf (1)}] ${\rm GExt}^i_{\Cadl}(X, S^m(M)) \cong {\rm GExt}^i_{R}(\frac{X_m}{B_m(X)}, M)$.

\item[{\bf (2)}] ${\rm GExt}^i_{\Cadl}(S^m(M), Y) \cong {\rm GExt}^i_{R}(M, Z_m(Y))$. 
\end{itemize}
\end{proposition}

One may be tempted to prove this result by considering, for example, exact Gorenstein-projective resolutions of $M$, say $\textbf{\textit{C}}_\bullet \longrightarrow M$. The problem is that the complex $S^m(\textbf{\textit{C}}_\bullet \longrightarrow M)$ is not necessarily ${\rm Hom}_{\Cadl}({\rm dw}\widetilde{\mathcal{GP}roj},-)$-exact. The proof we give below uses the fact that from a special Gorenstein-projective pre-cover of $X$, we can obtain a special Gorenstein-projective pre-cover of $\frac{X_m}{B_m(X)}$. Before going into the details, we need the following definitions and lemmas. \\

\begin{definition} A \underline{cotorsion pair} $(\mathcal{A,B})$ in an Abelian category $\mathcal{C}$ is given by two classes $\mathcal{A}$ and $\mathcal{B}$ of objects in $\mathcal{C}$ such that $\mathcal{A} = \mbox{}^\perp\mathcal{B}$ and $\mathcal{B} = \mathcal{A}^\perp$. Given a chain complex $(\mathcal{A,B})$ in $\mathcal{C}$, we say that a chain complex $X$ is a \underline{differential graded $\mathcal{A}$-complex} if $X_m \in \mathcal{A}$ for every $m \in \mathbb{Z}$, and if every chain map $X \longrightarrow B$ is chain homotopic to zero whenever $B$ is a $\mathcal{B}$-complex. The class of differential graded $\mathcal{A}$-complexes shall be denoted by ${\rm dg}\widetilde{\mathcal{A}}$. \\
\end{definition}

\begin{lemma} If $R$ is a Gorenstein ring, then ${\rm dw}\widetilde{\mathcal{GP}roj} = {\rm dg}\widetilde{\mathcal{GP}roj}$. Dually, ${\rm dw}\widetilde{\mathcal{GI}nj} = {\rm dg}\widetilde{\mathcal{GI}nj}$. 
\end{lemma}

\newpage

\begin{proof} On the one hand, ${\rm dw}\widetilde{\mathcal{GP}roj} = \mbox{}^\perp(\widetilde{\mathcal{W}})$ by \cite[Theorem 3.3.5]{GR}. On the other hand, \cite[Proposition 3.6]{Gil} implies that ${\rm dg}\widetilde{\mathcal{GP}roj} = \mbox{}^\perp(\widetilde{\mathcal{W}})$, since $(\mathcal{GP}roj, \mathcal{W})$ is a cotorsion pair and every module has a special Gorenstein-projective pre-cover.  
\end{proof}

\begin{proof}[Proof of Proposition \ref{isosgorspheres}] We only prove {\bf (1)}. Since ${\rm dw}\widetilde{\mathcal{GP}roj}$ is a special pre-covering class, there exists a short exact sequence $0 \longrightarrow W \longrightarrow C \longrightarrow X \longrightarrow 0$ where $W \in \widetilde{\mathcal{W}}$ and $C$ is a Gorenstein-projective complex. Using the fact that $X$ is exact and \cite[Lemma 3.2 (2)]{Perez}, we have a induced short exact sequence $0 \longrightarrow \frac{W_m}{B_m(W)} \longrightarrow \frac{C_m}{B_m(C)} \longrightarrow \frac{X_m}{B_m(X)} \longrightarrow 0$. On the one hand, $W \in \widetilde{\mathcal{W}}$ implies that $\frac{W_m}{B_m(W)} \cong Z_{m-1}(W) \in \mathcal{W}$. On the other hand, $X$ and $W$ are exact and so $C$ is also exact (the class of exact complexes is closed under extensions). We have $C \in {\rm dw}\widetilde{\mathcal{GP}roj} \cap \mathcal{E} = {\rm dg}\widetilde{\mathcal{GP}roj} \cap \mathcal{E} = \widetilde{\mathcal{GP}roj}$, where the last equality follows by \cite[Theorem 3.12]{Gil}. Then $\frac{C_m}{B_m(C)} \cong Z_{m-1}(C) \in \mathcal{GP}roj$. Hence $\frac{C_m}{B_m(C)} \longrightarrow \frac{X_m}{B_m(X)}$ is a special Gorenstein-projective pre-cover of $\frac{X_m}{B_m(X)}$.    

Note that $0 \longrightarrow W \longrightarrow C \longrightarrow X \longrightarrow 0$ is ${\rm Hom}_{\Cadl}({\rm dw}\widetilde{\mathcal{GP}roj},-)$-exact since $W \in \widetilde{\mathcal{W}}$. Similarly, $0 \longrightarrow \frac{W_m}{B_m(W)} \longrightarrow \frac{C_m}{B_m(C)} \longrightarrow \frac{X_m}{B_m(X)} \longrightarrow 0$ is ${\rm Hom}_R(\mathcal{GP}roj,-)$-exact. It follows by \cite[Theorem 12.1.4]{EJ} that there are long exact sequences \[ 0 \rightarrow {\rm Hom}(X,S^m(M)) \rightarrow {\rm Hom}(C,S^m(M)) \rightarrow {\rm Hom}(W,S^m(M)) \rightarrow {\rm GExt}^1(X,S^m(M)) \rightarrow \cdots \mbox{ \ and } \] \[ 0 \rightarrow {\rm Hom}_R\left(\frac{X_m}{B_m(X)},M\right) \rightarrow {\rm Hom}_R\left(\frac{C_m}{B_m(C)},M\right) \rightarrow {\rm Hom}_R\left(\frac{W_m}{B_m(W)},M\right) \rightarrow {\rm GExt}^1_R\left(\frac{X_m}{B_m(X)},M\right) \rightarrow \cdots \] where ${\rm GExt}^1_{\Cadl}(C,S^m(M)) = 0$ and ${\rm GExt}^1_R\left(\frac{C_m}{B_m(C)},M\right) = 0$. It follows that we have the following commutative diagram with exact rows:
\[ \begin{tikzpicture}
\matrix (m) [matrix of math nodes, row sep=2em, column sep=2em, text height=1.5ex, text depth=0.25ex]
{ 0 & {\rm Hom}(X,S^m(M)) & {\rm Hom}(C,S^m(M)) & {\rm Hom}(W,S^m(M)) & {\rm GExt}^1(X,S^m(M)) & 0 \\
0 & {\rm Hom}_R\left(\frac{X_m}{B_m(X)},M\right) & {\rm Hom}_R\left(\frac{C_m}{B_m(C)},M\right) & {\rm Hom}_R\left(\frac{W_m}{B_m(W)},M\right) & {\rm GExt}^1_R\left(\frac{X_m}{B_m(X)},M\right) & 0 \\ };
\path[->]
(m-1-1) edge (m-1-2) (m-1-2) edge (m-1-3) edge node[right] {$\cong$} (m-2-2) (m-1-3) edge (m-1-4) edge node[right] {$\cong$} (m-2-3) (m-1-4) edge (m-1-5) edge node[right] {$\cong$} (m-2-4) (m-1-5) edge (m-1-6) edge (m-2-5)
(m-2-1) edge (m-2-2) (m-2-2) edge (m-2-3) (m-2-3) edge (m-2-4) (m-2-4) edge (m-2-5) (m-2-5) edge (m-2-6);
\end{tikzpicture} \] 
By diagram chasing one has that the rightmost column is an isomorphism. The case $i > 1$ follows by induction. 
\end{proof}

%%%%%%%%%%%%%%%%%%%%%%%%%%%%%%%%%%%%%%%%%%%%%%%%%%%%%%%%%%%%%%%%%%%%%%%%%%%%%%%%%%%%%%%%%%
%%%%%%%%%%%%%%%%%%%%%%%%%%%%%%%%%%%%%%%%%%%%%%%%%%%%%%%%%%%%%%%%%%%%%%%%%%%%%%%%%%%%%%%%%%

\section*{Acknowledgements} 

The author wants to thank the financial support by the grants N000141310260 and FA95501410031 from the Office of Naval Research and the Air Force Office of Scientific Research, respectively.

%%%%%%%%%%%%%%%%%%%%%%%%%%%%%%%%%%%%%%%%%%%%%%%%%%%%%%%%%%%%%%%%%%%%%%%%%%%%%%%%%%%%%%%%%%
%%%%%%%%%%%%%%%%%%%%%%%%%%%%%%%%%%%%%%%%%%%%%%%%%%%%%%%%%%%%%%%%%%%%%%%%%%%%%%%%%%%%%%%%%%

\end{document}